\documentclass[11pt,english]{article}

\usepackage[english]{babel}
\usepackage{amsthm}
\usepackage{amsmath}
\usepackage{amssymb}
\usepackage{graphicx}
\usepackage{dsfont}
\usepackage[bf]{caption}
\usepackage{hyperref} 
\usepackage{color}

\usepackage{anysize}
\marginsize{2cm}{2cm}{2.5cm}{2.4cm}

\newtheorem{thm}{Theorem}[section]

\newtheorem{lem}[thm]{Lemma}
\newtheorem{prop}[thm]{Proposition}

\theoremstyle{definition}
\newtheorem{defi}[thm]{Definition}

\newtheoremstyle{rmk}
  {12pt}                   
  {12pt}                   
  {}                       
  {}                       
  {\normalfont\bfseries}   
  {.}                      
  { }               
  {}

\theoremstyle{rmk}

\newtheorem{rmk}[thm]{Remark}

\renewcommand{\qedsymbol}{$\blacksquare$}
\renewcommand{\labelenumi}{(\arabic{enumi})}
\renewcommand{\labelenumii}{(\alph{enumii})}

\newcommand{\p} {\textnormal{\textsf{P}}}
\newcommand{\E} { \textnormal{\textsf{E}}}
\newcommand{\N} { \mathbb{N} }
\newcommand{\Z} { \mathbb{Z} }
\newcommand{\R} { \mathbb{R} }

\newcommand{\V} { \mathbb{V} }
\newcommand{\new}{}

\newcommand{\1}[1]{{\mathds 1}_{\{#1\}}}
\newcommand{\I}{{\mathds 1}}

\newcommand{\bea}{\begin{eqnarray*}}
\newcommand{\eea}{\end{eqnarray*}}
\newcommand{\lp}[1] {\mathcal{L}(#1)}
\newcommand{\lpr}[1] {\mathcal{L^+}(#1)}
\newcommand{\lpl}[1] {\mathcal{L^-}(#1)}
\newcommand{\pc} {p_c}
\newcommand{\qc} {q_c}
\newcommand{\po} {p_0}
\newcommand{\qo} {q_0}
\newcommand{\mc} {m_c}
\newcommand{\mo} {m_0}
\newcommand{\fip} {\varphi_r}
\newcommand{\fim} {\varphi_{\ell}}
\newcommand{\mm} {m^*}

\allowdisplaybreaks

\begin{document}
\title{Cookie branching random walks}
 \author{Christian Bartsch$^{1}$ \and Michael Kochler$^{1}$ \and
Thomas Kochler$^{1}$ \and Sebastian M\"{u}ller$^{2}$ \and Serguei~Popov$^{3}$}
\date{\today}
\maketitle

{\footnotesize \noindent $^{~1}$
Center for Mathematical Sciences, Technische Universit\"at M\"unchen,
Boltzmannstra\ss e 3, D-85748 Garching, Germany\\
\noindent e-mail: \texttt{christian.bartsch@ma.tum.de, michael.kochler@ma.tum.de, thomas.kochler@ma.tum.de}

\noindent $^{~2}$LATP, Aix-Marseille Universit\'e,
39 rue Joliot Curie, 13453 Marseille cedex 13, France\\
\noindent e-mail: \texttt{mueller@cmi.univ-mrs.fr},
url: \texttt{http://www.latp.univ-mrs.fr/$\sim$mueller}

\noindent $^{~3}$Department of Statistics, 
Institute of Mathematics, Statistics and Scientific Computation,
University of Campinas--UNICAMP, rua S\'ergio Buarque de Holanda 651, 13083--859, Campinas SP,
Brazil \\
\noindent e-mail: \texttt{popov@ime.unicamp.br}, 
url: \texttt{http://www.ime.unicamp.br/$\sim$popov}

}

\begin{abstract}
{\new We consider  branching random walks in discrete time where at each time unit particles first  produce  offspring  and thereafter the produced offspring particles move independently according to some nearest neighbour random walk. In our model, the branching random walks live  on~$\Z$ and the particles behave differently in visited and unvisited sites.} 
Informally, each site on the positive half-line
contains initially a cookie.  On the first visit of a site its cookie is removed and 
particles at positions with a cookie reproduce and move
differently from particles on sites without cookies. 
Therefore, the movement and the reproduction of the
particles depend on the previous behaviour of the population of
particles. {\new We give a classification in recurrence and transience, i.e.}~whether infinitely many particles visit the origin
or not.
\end{abstract}

\section{Introduction and results}
\label{s_intro}
In the recent years a lot of attention was attracted by the model called \emph{excited random walk}, which can be informally described in the following way. It is a process that depends on the past through the set of visited sites: {\new  the random walker's movement in a state space (usually $\Z^d$ for $d\ge1$) at time $n\in\N_0$ depends on whether the random walker has already visited its current position before time $n$. } Such a model was introduced in~\cite{BW} {\new  and studied in numerous subsequent papers. We refer for example to~}\cite{BS08, KZ08, Z05} (one-dimensional case, where, as usual, more complete results are available), \cite{BS09, BR07, HH09, MPRV} ({\new  for the } multi-dimensional case and trees), and {\new  the} references therein. This model is also frequently called \emph{cookie random walk} {\new  the idea being } that initially all sites contain (one or several) cookies {\new  which are consumed by the random walker at the time of its first visit of the respective site. Whenever the random walker consumes a cookie at some site, this changes the transition probabilities at this site} (usually by giving {\new  the random walk } a drift in some direction).

In this paper we {\new  adopt the idea of having consumable cookies at certain sites to branching random walks. Hence, } we consider not only {\new  one single random walker or one particle } that walks around {\new  in a state space}, but {\new  a whole population or cloud of particles which independently produce offspring particles according to given offspring distributions. Thereafter, the newly produced particles move independently according to given transition probabilities. }
The transition and branching parameters depend on whether the site was visited before or not. {\new  More precicely, using the above ``cookie'' interpretation, it can be pictured} that initially each site contains a cookie which is removed when at least one particle visits the site. {\new  Thus,} we call our model cookie branching random walk (CBRW).

Different kinds of models related to branching random walks recently appeared in the literature; we refer to~\cite{CP, CP07, CometsYoshida, HuShi, Mueller}. As far as we know, however, the situation when the behaviour of the branching random walk is changed in the visited sited was previously not considered.

However, it is interesting to note that there is a model that lies in some sense inbetween the excited random walk and the CBRW. This model is usually called \emph{frog model} (we cite for example~\cite{AMPR, CQR}), and it can be described in the following way: the particles do not branch in already visited sites, and when one or several particles visit a new site, \emph{exactly one of them} is allowed to branch. (Another interpretation is that initially every site contains a number of sleeping {\new frogs} and an active {\new frog} is placed somewhere; when an active {\new frog  jumps on top of sleeping frogs, } those are activated too.)

Let us now turn to the formal description of the CBRW. First, we have to choose the initial configuration of the cookies. In this paper we restrict ourselves to the case in which we have one cookie at every non-negative integer and no cookies at the negative 
integers. Thus, {\new  if $c_n(x)$ denotes the number of cookies at position $x\in\Z$ at time $n\in\N_0$, the cookie configuration as described above is given by }
\[
 c_0(x):=\begin{cases}
        1,& \text{if } x\ge0,\\
        0,&  \text{if } x<0.
         \end{cases}
\]
As it turns out, the above configuration of cookies 
is a natural choice for an initial configuration in order to point out the
 essential differences in the evolution of the process. In particular,
 further results for the initial configuration
$(c_0(x)=1\text{ for all }x\in\Z)$
can be derived easily (cf.\ Section~\ref{section6}).
At time 0 the CBRW starts with one initial particle at the origin. 
To specify the evolution of the population of particles, we need 
the following ingredients:
\begin{itemize}
 \item \textit{the cookie offspring distribution}
$\mu_c=\Big(\mu_c(k) \Big)_{k\in\N_0}$ with mean
$\mc:= \sum_{k=1}^{\infty} k \mu_c(k)$;
 \item \textit{the cookie transition probabilities}
$\pc \in(0,1)$, $\qc:=1-\pc$;
 \item \textit{the no-cookie offspring distribution}
$\mu_0=\Big(\mu_0(k)\Big)_{k\in\N_0}$ with mean
$\mo:= \sum_{k=1}^{\infty} k \mu_0(k)$;
 \item \textit{the no-cookie transition probabilities}
$\po \in(0,1)$, $\qo:=1-\po$.
\end{itemize}
We say a particle produces offspring according to an offspring distribution $\mu=(\mu(k))_{k\in\N_0}$ if the probability of having $k$ offspring is $\mu(k)$. {\new  Having fixed the above quantities, } the population of particles evolves at every discrete time unit $n\in\N_0$ according to the following rules:
\begin{enumerate}
 \item
First, every existing particle produces offspring independently of the other particles. Each particle either reproduces according to the offspring distribution $\mu_c$ if there is a cookie at its position or according to~$\mu_0$ otherwise. After that the parent particle dies.
\item
Secondly, after the branching the newly produced offspring particles move independently of each other either one step to the right or one step to the left. Again the movement depends on whether the particles are at a position with or without a cookie. If there is a cookie, each particle moves to the right (left) with probability $\pc$ $\left(\qc\right)$. Otherwise, if there is no cookie, the transition probabilities are given by~$\po$ and~$\qo$.
\item 
Finally, each cookie which is located at a position where at least one particle has produced offspring is removed. We note that different particles share the same cookie if they are at a position with a cookie at the same time. Moreover, due to the chosen initial configuration of the cookies only the leftmost cookie can be consumed at every time step.
\end{enumerate}

We now introduce some essential notations and assumptions. 
Since we do not want the process to die out, we assume that
\[
\mu_c(0)=\mu_0(0)=0
\]
holds. Further, to avoid additional technical difficulties, 
we suppose that we have
\begin{equation} \label{assm}
M:=\sup\left\{k\in\N_0:\mu_c(k)+\mu_0(k)>0\right\}<\infty.
\end{equation}
In fact, we believe that the results remain true if we replace~\eqref{assm} 
by the assumption that the cookie and the no-cookie offspring variance is 
finite.
In the following we want to distinguish different particles of the CBRW by using the usual Ulam-Harris labelling. 
Therefore, we enumerate the offspring of every particle and introduce the set
\[
\V:=\bigcup_{n\in\N_0}\N^n
\]
as the set of all particles which may be produced at some time in the whole process. Here $\N^{0}:=\{\emptyset\}$ is defined as the set containing only the root of the tree. In this setting, $\nu =(\nu_1,\nu_2,\ldots,\nu_n) \in \V$ labels the particle which is the $\nu_n$-th offspring  of the particle $(\nu_1,\nu_2,\ldots,\nu_{n-1})$. By iteration we can trace back the ancestral line of $\nu$ to the initial particle $\emptyset$. Further, the generation (length) of the particle $\nu\in\V$ is denoted by~$|\nu|$,
and for two particles $\nu,\eta\in\V$ we write $\nu\succ\eta\ (\text{respectively, }\nu \succeq \eta)$ if~$\nu$ is a descendant of the particle~$\eta$ (respectively, if~$\nu$ is a descendant of~$\eta$ or~$\eta$ itself). We use the same notation $\nu\succ U$ (respectively, $\nu\succeq U$) for some set $U\subset\V$ if there is a particle $\eta\in U$ with $\nu\succ\eta$ (respectively, $\nu\succeq\eta$). 
{\new The position of a particle is the place where it jumps to, i.e.~the position at the end of step (2) above.}
With the these notations, we can consider the actually produced particles in the CBRW. For $n \in \N_0$ and $x \in \Z$ let $Z_n(x) \subset \N^n \subset \V$ denote the random set of particles which are at position~$x$ at time~$n$. Thus
\[
Z_n:=\bigcup_{x\in\Z}Z_n(x)
\]
is the set of all particles which exist at time~$n$ and using this we can define ${\mathcal Z}:=\bigcup_{n\in\N_0}Z_n$ as the set of all particles ever produced. Then, for every particle $\nu\in {\mathcal Z}$ we write~$X_\nu$ for its random position in~$\Z$ and the collection of all positions of all particles $(X_\nu)_{\nu\in{\mathcal Z}}$ is what we call CBRW. Further, the position of the leftmost cookie is denoted by
\begin{align*}
l(n):=&\ \min\{x\in\N:c_n(x)=1\}.
\end{align*}

Now, we are able to define the set of particles~$\lp{n}$ which is crucial for our considerations:
\begin{align*}
\lp{n}:=&\ Z_n(l(n)).
\end{align*}
The particles that belong to $\lp{n}$ are located at the position of the leftmost cookie and thus they are the only particles which produce offspring according to $\mu_c$. We call the process $\big(\lp{n}\big)_{n\in\N_0}$ \emph{leading process} (and use the abbreviation LP) since it contains the rightmost particles if $\lp{n}\neq\emptyset$. One key observation for the understanding of the CBRW is that the particles in the LP constitute a Galton-Watson process (GWP) as long as there are particles 
in the LP. The associated mean offspring is given by~$\pc\mc$ and thus we call the LP supercritical (respectively, subcritical, or critical) if~$\pc\mc$ is greater than~$1$ (respectively, smaller than~$1$, or equal to~$1$).

As it is usually done in the context of branching random walks (BRW), we now define three different regimes:
\begin{defi}
\label{def_rec_trans}
A CBRW is called
\begin{enumerate}
\item
\textit{strongly recurrent} if a.s.\ infinitely many particles visit the origin, i.e.\ $\p\left(|Z_n(0)|\xrightarrow[n\to\infty]{} 0\right)=0$,
\item
\textit{weakly recurrent} if $\p\left(|Z_n(0)| \xrightarrow[n\to\infty]{} 0\right)\in(0,1)$,
\item
\textit{transient} if $\p\left(|Z_n(0)| \xrightarrow[n\to\infty]{}0\right)=1$.
\end{enumerate}
\end{defi}
We mention that these regimes may have different names in the literature; for instance, strong local survival, local survival, and local extinction of~\cite{GMPV} correspond to strong recurrence, recurrence, and transience of the present paper. The transient regime may be subdivided into \emph{transient to the left} (resp. \emph{transient to the right}) if the negative (resp. positive) integers are visited infinitely many times. Criteria for the recurrence/transience behaviour of BRW are well-known in the literature. In our setting the BRW of interest is the process related to the behaviour of the particles without cookies. In the following we call this process BRW without cookies. It is a BRW in the usual sense started with one particle at 0, with offspring distribution~$\mu_0$ and transition probabilities~$\po$, $\qo$ to the nearest neighbours. {\new  For this process we have the following proposition that goes back to fundamental papers by Biggins \cite{Bi}, Hammersley \cite{Ha}, and Kingman \cite{Ki}; for a proof we refer to Theorem~18.3 in~\cite{peres} and Theorem~3.2 in~\cite{mueller}.}
\begin{prop}The BRW without cookies is~\label{prop1.11}
\begin{enumerate}
\item transient to the right iff 
\[
\po > \frac12\quad\textnormal{and}\quad\mo\le\frac{1}{2\sqrt{\po \qo}},
\]
\item transient to the left iff 
\[
\po < \frac12\quad\textnormal{and}\quad\mo\le\frac{1}{2\sqrt{\po \qo}},
\]
\item and strongly recurrent in the remaining cases.
\end{enumerate}
\end{prop}

In the transient cases, we define
\[
 \fim = {\new
         \frac{1}{2\po\mo}  \left(1 - \sqrt{1-4\po\qo\mo^2}\right).
   }
\]
{\new We note that $\fim$ reduces to $\min\{1 , \tfrac{\qo}{\po} \}$ if we assume $\mo =1$}. An interpretation of the quantity~$\fim$ is given  in Section~\ref{s_prelim} below.

Now, we are ready to formulate the main results of this paper, which give the classification of the process with respect to weak/strong recurrence in the sense of Definition~\ref{def_rec_trans}.
\renewcommand{\labelenumi}{(\alph{enumi})}
\renewcommand{\labelenumii}{(\roman{enumii})}
\begin{thm}
\label{thm1}
Suppose that the BRW without cookies is transient to the right.
\begin{enumerate}
\item
If the LP is supercritical, i.e.\ $\pc\mc > 1$ holds, then
\begin{enumerate}
\item
the CBRW is strongly recurrent iff $\pc\mc  \fim \ge 1$,
\item
and the CBRW is transient {\new to the right} iff $\pc\mc \fim < 1$.
\end{enumerate}
\item
If the LP is subcritical or critical, i.e.\ $\pc\mc \leq  1$ holds, 
then the CBRW is transient {\new to the right}.
\end{enumerate}
\end{thm} 

\begin{thm}\label{thm2}
Suppose that the BRW without cookies is strongly recurrent.
Then the CBRW is strongly recurrent, no matter whether the LP
 is subcritical, critical or supercritical.
\end{thm} 
\begin{thm}\label{thm3}
Suppose that the BRW without cookies is transient to the left.
\begin{enumerate}
\item
If the LP is supercritical, i.e.\ $\pc\mc > 1$ holds, 
then the CBRW is weakly recurrent.
\item
If the LP is critical or subcritical, i.e.\ $\pc\mc\le1$ holds, 
then the CBRW is transient {\new to the left}.
\end{enumerate}
\end{thm}

\section{Preliminaries}
\label{s_prelim}
Analogously to the notation which we use for the CBRW let $\big(Y_{\nu}\big)_{\nu \in \mathcal{Y}}$ denote the BRW without cookies. Here $\mathcal{Y}$ denotes the set of all particles ever produced and (for every $\nu \in \mathcal{Y}$) $Y_{\nu}$ denotes the random position of the particle $\nu$. We define $\Lambda^+_0=\Lambda^-_0:=1$, and
\begin{align}
\Lambda^+_n:={\new \left|\left\{\nu\in\mathcal{Y}:\ Y_\nu=n,\ 
Y_\eta<n\ \forall \eta\prec\nu \right\}\right|}, \quad
\Lambda^-_n:={\new \left|\left\{\nu\in\mathcal{Y}:\  Y_\nu=-n,\ Y_\eta>-n\ 
\forall \eta\prec\nu \right\}\right|} \label{deflam}
\end{align}
for $n\in\N$. Here $\Lambda_n^+$ (respectively, $\Lambda_n^-$) denotes the random number of particles which are the first in their ancestral line to reach the position $n$ (respectively, $-n$). In addition, we define
\begin{align} \label{deffi}
 \fip:=\E[\Lambda^+_1], \qquad
 \fim:=\E[\Lambda^-_1].
\end{align}
Note that we have
\[
\p\big(\Lambda_1^+<\infty\big) = \p\big(\Lambda_1^-<\infty\big) = 1
\]
if the BRW without cookies $\big(Y_{\nu}\big)_{\nu \in \mathcal{Y}}$ is transient. In this case the processes $\big(\Lambda_n^+\big)_{n\in\N_0}$ and $\big(\Lambda_n^-\big)_{n\in\N_0}$ are both GWPs. An important observation is that $\fip$ and $\fim$ can be expressed using the first visit generating function of the underlying random walk. Thus, denote by $X_{n}$ the nearest neighbour random walk defined by $\p(X_{n+1}=x+1\mid X_{n}=x)=p_{0}$ and $\p(X_{n+1}=x-1\mid X_{n}=x)=q_{0}$. The \emph{first visit generating function} is defined by
\[ F(x,y| z) = \sum_{n=0}^{\infty} \p(X_{n}=y, X_{k}\neq y~ \forall k<n\mid X_{0}=x)z^{n}.\] 
A (short)  thought reveals that $\fip = F(0,1|m_{0})$ and $\fim = F(0,-1|m_{0})$ and standard calculations yield the following formulas; for both arguments we also refer to Chapter 5 in \cite{Woe}.

\begin{prop} \label{prop1.11a} If the BRW without cookies is transient, 
we have
\begin{align}
\fip = \frac{1}{2\qo\mo}  \left(1 - \sqrt{1-4\po\qo\mo^2}\right), 
\quad\text{and} \quad
\fim =\frac{1}{2\po\mo}  \left(1 - \sqrt{1-4\po\qo\mo^2}\right). 
\label{prop1.11a.r2}
\end{align}
\end{prop}

\begin{rmk}
A natural special case is the situation where $\mu_{0}(1)=1$ {\new (and $m_0=1$)}. In this model particles can only branch at positions with a cookie. In sites without cookies the process reduces to an asymmetric random 
walk $\big(Y_n\big)_{n\in\N_0}$ on $\Z$ with transition 
probabilities $\po$ and $\qo$. 
Here $\fip$ and $\fim$ simplify to the probabilities of an asymmetric 
random walk to ever reach $+1$ or $-1$, respectively, i.e.
\begin{align} \label{prop1.11a.r2.2}
\fip= \p \big(\exists n \in \N:\ Y_n = +1 \big) 
= \min \left\{1, \tfrac{\po}{\qo} \right\} \quad \text{and}\quad
\fim= \p \big(\exists n \in \N:\ Y_n = -1 \big) 
= \min \left\{1, \tfrac{\qo}{\po} \right\}.
\end{align} 
\end{rmk}


Next, we collect some known facts about Galton-Watson processes that
will be needed in the sequel.
An important tool for the proofs is to identify GWPs which are 
embedded in the CBRW. For the rest of this paper the processes
\[
\big(GW^{\text{super}}_n \big)_{n \in \N_0}, 
\big(GW^{\text{sub}}_n \big)_{n \in \N_0} \text{ and } 
\big(GW^{\text{cr}}_n \big)_{n \in \N_0}
\]
shall denote a supercritical, subcritical or critical GWP 
started with $z\in \N$ particles with respect to the probability 
measure~$\p_z$. Furthermore, let $T^{\text{super}}, 
T^{\text{sub}} \text{ and } T^{\text{cr}}$
denote the time of extinction corresponding to the above GWPs, i.e.
\[
T^{\text{super}}:= \inf \{n\ge 0: \ GW^{\text{super}}_n=0 \}
\]
and analogously for the subcritical and critical case.
\begin{prop} \label{propo1}
For a subcritical GWP $\big(GW^{\textnormal{sub}}_n \big)_{n \in \N_0}$ 
with strictly positive and finite offspring variance there is a 
constant $c>0$ such that
\begin{equation*}
\lim_{n \to \infty} 
\frac{\p_1 \left( GW_n^{\textnormal{sub}} > 0  \right)}
{ \E_1 \left[ GW_1^{\textnormal{sub}} \right]^n} = c.
\end{equation*}
\end{prop}For a proof see for instance Theorem~2.6.1 in~\cite{Jagers}. 

\begin{prop} \label{propo2}
For a critical GWP $\big(GW^\textnormal{cr}_n \big)_{n \in \N_0}$ 
 with strictly positive and finite offspring variance there is a 
constant $c>0$ such that
\begin{equation*}
\lim_{n \to \infty} n\, \p_1 \left( GW_n^\textnormal{cr} > 0  \right) = c.
\end{equation*}
\end{prop}
For a proof see for instance Theorem I.9.1 in \cite{Athreya/Ney}. Using the  inequality $1 - x \leq \exp(-x)$ we obtain the following  consequence of Proposition \ref{propo2}.

\begin{prop} \label{propo3}
For the extinction time $T^\textnormal{cr}$ of a critical 
GWP with strictly positive and finite offspring variance there exists a
 constant $C > 0$ such that
\begin{equation*}
\p_z \big( T^\textnormal{cr} \le n \big) 
\leq \exp \left( - C \frac{z}{n} \right) 
\end{equation*}
for all $n \in \N$ and for all $z\in\N$.
\end{prop}

\begin{prop} \label{propo4}
For the extinction time $T^\textnormal{cr}$ of 
a critical GWP with strictly positive and finite offspring variance
 there exists a constant $C> 0$ such that
\begin{equation*}
\p_z \big( T^\textnormal{cr} = n \big) \leq C \frac{z}{n^2}
\end{equation*}
for all $n \in \N$ and for all $z\in\N$.
\end{prop}
\begin{proof}
Due to Corollary I.9.1 in \cite{Athreya/Ney} (with $s=0$), 
there is a constant $c>0$ such that
\[
\lim_{n \to \infty} n^2 \p_1 \big( T^\text{cr} = n+1 \big) = c.
\]
Therefore, we get for $n \in \N$
\begin{align*}
\p_z \big( T^\text{cr} = n  \big) \leq z \, \p_1 
\big( T^\text{cr} = n \big) = z \frac{1}{(n-1)^2} 
\big( c + o(1) \big) \leq C  \frac{z}{n^2} 
\end{align*}
for a suitable constant $C > 0$.
\renewcommand{\qedsymbol}{$\square$}
\end{proof}

{\new
\begin{lem}\label{lemxx}
Let us consider a BRW $(Y_{\nu})_{\nu \in \mathcal{Y}}$ without cookies with parameters $\mu_0, \po, \qo$ (and start in 0 with one particle). If the BRW is transient to the right, we have for $n \in \N$
\begin{equation*}
\p\left(\exists \nu \in \mathcal{Y}:\ Y_{\nu} = -n \right) = \big(c + o(1)\big)(\varphi_\ell)^n
\end{equation*}
for some constant $c > 0$, where $\lim\limits_{n\to\infty} o(1) = 0$.
\end{lem}
\begin{proof}
We consider the process $(\Lambda^-_n)_{n \in \N_0}$ introduced in \eqref{deflam} and observe that this process is a GWP with mean $\varphi_\ell < 1$ due to \eqref{prop1.11a.r2} and \eqref{prop1.11a.r2.2}. Using condition \eqref{assm}, it is not difficult to verify that we have $\E[(\Lambda^-_1)^2] < \infty$. Therefore, Proposition \ref{propo1} completes the proof.
\renewcommand{\qedsymbol}{$\square$}
\end{proof}
}

\section{Proofs of the main results}
{\new  We use both of the symbols $\blacksquare$ and $\square$ to signal the completion of a proof. The symbol $\blacksquare$ is used at the end of the proofs of the major results; whereas $\square$ is used for the proofs of auxiliary results which are part of another proof.}
\subsection*{Proof of Theorem~\ref{thm1}}
\textbf{Proof of part (a).}\\
In this part of the proof we suppose $\pc\mc>1$, i.e.\ 
the LP is supercritical. 
For $n\in\N$ we define inductively the $n$-th extinction time and 
the $n$-th rebirth time of the LP by
\begin{align*} 
\tau_n\,:=&\ \inf\big\{i>\sigma_{n-1}:\, |\lp{i}|=0\big\},\\
\sigma_n:=&\ \inf\big\{i>\tau_n:\, |\lp{i}|\geq1\big\}
\end{align*}
with $\sigma_0:=0$ and $\inf\,\emptyset:=\infty$. Since $p_{0}>1/2$ and  the LP is supercritical we have that $\p(\sigma_{n}<\infty\mid \tau_{n}<\infty)=1$ and $\p(\tau_{n+1}=\infty\mid \tau_{n}<\infty)\geq \p(\tau_{1}=\infty)>0$ for all $n\geq 0$. Hence, 
we a.s.\ have
\begin{equation} \label{lem5.2.1}
\sigma^*:=\inf\{n\in\N_0:|\lp{i}|\ge1\ \forall\ i\ge n  \}<\infty.
\end{equation}
It is a well-known fact that conditioned on survival a supercritical GWP with finite second moment normalized by its mean converges to a strictly positive random variable (e.g. see Theorem~I.6.2 in~\cite{Athreya/Ney}). Considering the LP {\new separately} on the events $\{\sigma^{*}=k\}$ for $k\in \N_{0}$ yields
\begin{equation} \label{lem5.2.2}
\lim_{n\to \infty} \frac{|\lp{n}|}{(\pc\mc)^n} = W > 0
\end{equation}
for a strictly positive random variable $W$.  

 Now, we prove part~(i) of Theorem~\ref{thm1}(a).
Suppose that $\pc\mc\fim \geq 1$. 
For $n \in \N_0$, let us introduce
\[
L_n:=\left\{\nu\in Z_{n+1}(l(n)-1): \nu\succ \lp{n}\right\}.
\]
The set $L_n$ contains all particles that are produced in the LP at time $n$ and then leave the LP to the left. Thus  they are located at the position $l(n)-1$ at time $n+1$. Define the {\new events}
$
A_n:=\big\{ \exists\ \nu\succeq L_n: X_\nu=0 \big\}
$
for $n \in \N_0$.  In   order to show strong recurrence of the CBRW it is now sufficient {\new to prove} that 
\begin{equation}
\label{thm1.0}
 \p\Big(\limsup_{n\to\infty} A_n\Big)=1.
\end{equation}
As a first step to achieve this, we consider the events $
B_n:=\left\{ |L_n|\geq \big(\pc\mc)^n n^{-1},\ n\geq\sigma^* \right\}
$
for $n \in \N_0$ and show that
\begin{equation}
\label{thm1.1a}
\p\left(\liminf_{n\to\infty}B_n\right)=1.
\end{equation}
This provides a lower bound for the growth of $|L_n|$ for large~$n$. 
To see that~\eqref{thm1.1a} holds, 
we define $C_n:=\big\{ |\lp{n}|\geq (\pc\mc)^n n^{-1/2}\big\}$ 
and notice that due to \eqref{lem5.2.2} we have
\begin{equation}\label{thm1.1b}
\p\left(\liminf_{n\to\infty} C_n \right)=1.
\end{equation}
We observe that, given the event $C_n$, the random variable $|L_n|$ 
can be bounded from below by a random sum of 
$\lceil(\pc\mc)^n n^{-1/2}\rceil$ i.i.d.\ 
Bernoulli random variables with success probability~$\qc$. Hence, we can use a standard large deviation bound to see that $ \p( |L_n|< (\pc\mc)^n n^{-1}\ \mid\ C_n)$ decays exponentially in $n$. An application of the  Borel-Cantelli lemma now yields 
\begin{equation} 
\label{thm1.1b2.1}
\p\left(\limsup_{n\to\infty} \Big(\big\{ |L_n|< (\pc\mc)^n n^{-1}
 \big\}\cap C_n\Big) \right)=0.
\end{equation}
Since $\sigma^*< \infty$ a.s., 
\eqref{thm1.1b2.1} together with~\eqref{thm1.1b} yields~\eqref{thm1.1a}. 

Observe that on $\{n \ge \sigma^*\}$ the {\new number} of {\new  descendants } of every particle in $L_n$ which ever {\new reaches the position} $1,2,\ldots$ steps to the left for the first time in their genealogy constitutes an embedded GWP in the CBRW. Its mean is given by $\fim$, where $\fim < 1$ holds since the BRW without cookie is transient to the right (cf.\ {\new \eqref{prop1.11a.r2} and \eqref{prop1.11a.r2.2}}). 
 Using  {\new Lemma \ref{lemxx}} we therefore get
\begin{align}
\label{conj2}
 \p\big(\,A_n \mid  {\new B_n}\, \big)
\ge&\ 1-\left( 1- c(\fim)^n   
\right)^{(\pc\mc)^n n^{-1}}\nonumber 
\\
\ge&\ 1-\exp\left(- c (\fim)^n  
 (\pc\mc)^n n^{-1}\right)\nonumber 
\\
\ge&\ 1-\exp\left(- \tfrac{c}{n}\right)\nonumber\\
\ge&\ \tfrac{C}{n} 
\end{align}
for some $c, C>0$. Here we use that the position of 
a particle $\nu\in L_n$ is bounded by~$n$ (in fact by $n-1$).
Notice also that we have $\pc\mc \fim\ge1$ by assumption.
Since $\mathds{1}_{B_n}$ is measurable with respect to the $\sigma$-algebra 
generated by $|L_n|$ and $\sigma^*$, we have for $i,j\in\N$ with $i<j$
\begin{align}
 \p\left(\bigcap_{n=i}^j \big(A_n^{\textnormal{c}}\cap B_n\big) \right)
=\ &\E  \left[  \E\left[\left. \prod_{n=i}^j  
\mathds{1}_{A_n^{\textnormal{c}}\cap B_n}\ \right|\ |L_i|,\ldots,|L_j|,\ 
\sigma^*\right] \right]\displaybreak[0]\nonumber\\
=\ &\E  \left[  \Big(\prod_{n=i}^j \mathds{1}_{B_n}\Big)
\mathds{1}_{\{i\ge\sigma^*\}}\E\left[\left.\prod_{n=i}^j 
\mathds{1}_{A_n^{\textnormal{c}}}\ \right|\ |L_i|,\ldots,|L_j|,\ 
\sigma^*\right] \right] \nonumber\\
=\ & {\new \E  \left[  \prod_{n=i}^j \mathds{1}_{B_n} \E\left[\left. 
\mathds{1}_{A_n^{\textnormal{c}}}\ \right|\ |L_n|,\ \sigma^*\right] \right].} \label{conj2.2}
\end{align}
For the last step, we observe that on $\{i\geq\sigma^*\}$ 
the random variables $\big(\mathds{1}_{A_n^{\textnormal{c}}}\big)_{i\le n\le j}$
 are conditionally independent given $|L_i|,\ldots,|L_j|$ and $\sigma^*$. 
This holds because on $\{i\geq\sigma^*\}$ all the particles in 
$\bigcup_{n=i}^jL_n$ start independent BRWs which cannot reach the cookies 
anymore.
For the same reason on $\{i\geq\sigma^*\}$ each of the random variables 
$(\mathds{1}_{A_n^{\textnormal{c}}})_{i\le n}$ is conditionally 
independent of $(|L_k|)_{k\not=n}$ given $|L_n|$ and $\sigma^*$. 
With the help of~\eqref{conj2} and~\eqref{conj2.2} we can now conclude 
that we have
\begin{align}\label{conj3}
  \p\left(\bigcap_{n=i}^j \big(A_n^{\textnormal{c}}\cap B_n\big) 
\right)
 \le\ &{\new \prod_{n=i}^j  \left(1- \tfrac{C}n \right) \xrightarrow[j\to\infty]{}0.}
\end{align}
Therefore, for all $i\in\N$ we have 
$\p\left(\cap_{n=i}^\infty \big(A_n^{\textnormal{c}}\cap B_n\big) \right)=0$,
which implies
\begin{equation}
\label{conj3.1}
\p\left(\liminf_{n\to\infty} \big(A_n^{\textnormal{c}}\cap B_n\big) \right)=0.
\end{equation}
Since \eqref{thm1.1a} holds, \eqref{conj3.1} yields
$\p\left(  \liminf_{n\to\infty} A_n^{\textnormal{c}}\right)=0$.
Thus, we have established~\eqref{thm1.0} and so (i) of Theorem~\ref{thm1}(a) is proven.

\medskip

Next, we prove part~(ii) of Theorem~\ref{thm1}(a).
Suppose that $\pc\mc \fim < 1$. For sake of simplicity we assume $\sigma^{*}=0$. The proof is analogous for $\sigma^{*}=k$ for $k\in \N$. The idea of the proof is to show that the expected number of particles that visit the origin the second time (the first time after time 0) in their genealogy is finite. Since the BRW without cookies is transient this implies transience of the CBRW. We note that no descendant of a particle that visited the origin after time $0$ can ever reach a cookie again since $\sigma^{*}=0$. (In the case  $\sigma^{*}=k$ only a finite number of particles that have visited the origin up to time $k$ can have descendants which reach a cookie again.) More formally, define
\[ 
\Gamma_{n}={\new \left|\left\{\xi\in \mathcal{Z}:\ \xi \succeq L_n,\ X_{\xi}=0,\, X_{\omega} \neq 0\,  \forall \omega:\, \xi \succ \omega\succeq L_n    \right\}\right|}.
\]
Taking expectation yields
\[
\E[\Gamma_{n} \I_{\{\sigma^*=0\}}]= \E[{\new |L_{n}|}\I_{\{\sigma^*=0\}}] F(n,0|m_{0})\le (\pc \mc)^{n}\qc \mc(\fim)^{n},
\]
and thus we have $\E[\sum_{n} \Gamma_{n} \I_{\{\sigma^*=0\}}]<\infty$ since $\pc \mc \fim <1$.
Therefore, we can finally conclude that a.s.\ only finitely many 
particles visit the origin, i.e.\ the CBRW is transient. 
This completes the proof of part (a). $\hfill\blacksquare$

\medskip
\noindent
\textbf{Proof of part (b)}\\
In this part of the proof we suppose that the LP is subcritical or critical,  i.e.\ that $\pc\mc\leq1$. We start with Lemma~\ref{crclem1}, which states that except for finitely many times the particles at a single position $x\in\Z$ produce an amount of offspring which is close to the expected amount as long as there are many particles at this position. To do so, we first split the set of particles $Z_n(x)$ into the following two sets
\begin{align*}
Z_{n+1}^+(x)&:=\{\nu \in Z_{n+1}(x):\, \nu \succ Z_n(x-1)\},\\
Z_{n+1}^-(x)&:=\{\nu \in Z_{n+1}(x):\, \nu \succ Z_n(x+1)\}
\end{align*}
containing the particles which have moved to the right or to the 
left from time $n$ to time $n+1$. For $\varepsilon>0$, 
which we specify later (cf.\ \eqref{crc7} and~\eqref{crcchoicee2}), 
we  introduce the following sets:
\begin{align} \label{crc1a}
D^+_n(x) &:= \{x < l(n),\, |Z_n(x)|\geq n\} \cap 
\left(\left\{ \frac{|Z_{n+1}^+(x+1)|}{|Z_n(x)|} < (\po\mo-\varepsilon) 
\right\}\right.\nonumber\\
&\hspace{5cm}\left.\cup \left\{ (\po\mo+\varepsilon) 
< \frac{|Z_{n+1}^+(x+1)|}{|Z_n(x)|}  \right\} \right) ,
\nonumber \allowdisplaybreaks[1]\\
D^-_n(x) &:= \{x < l(n),\, |Z_n(x)|\geq n\} \cap 
\left(\left\{ \frac{|Z_{n+1}^-(x-1)|}{|Z_n(x)|} 
< (\qo\mo-\varepsilon) \right\}\right.\nonumber\\
&\hspace{5cm}\left.\cup \left\{ (\qo\mo+\varepsilon) 
< \frac{|Z_{n+1}^-(x-1)|}{|Z_n(x)|}  \right\} \right) ,
\nonumber \allowdisplaybreaks[1]\\
E^+_n &:= \{\lp{n}\geq n\} \cap \left(\left\{ \frac{|\lp{n+1}|}{|\lp{n}|} 
< (\pc\mc-\varepsilon) \right\}\cup \left\{ (\pc\mc+\varepsilon) 
< \frac{|\lp{n+1}|}{|\lp{n}|}  \right\} \right) ,
\nonumber \allowdisplaybreaks[1]\\
E^-_n &:= \{\lp{n}\geq n\} \cap 
\left(\left\{ \frac{|Z_{n+1}^-(l(n)-1)|}{|\lp{n}|} 
< (\qc\mc-\varepsilon) \right\}\cup \left\{ (\qc\mc+\varepsilon) 
< \frac{|Z_{n+1}^-(l(n)-1)|}{|\lp{n}|}  \right\} \right) ,\nonumber \\
F_n&:=  E^+_n \cup E^-_n \cup \bigcup\limits_{x\in\Z} 
\Big(D^+_n(x) \cup D^-_n(x)\Big).
\end{align}

\begin{lem} \label{crclem1}
{\new For all $\varepsilon > 0$}, we have
\begin{equation} \label{crc1}
\p \left(\limsup_{n\to\infty} F_n \right) = 0.
\end{equation}
\end{lem}
\begin{proof}[Proof of Lemma~\ref{crclem1}]
A large deviation estimate (note that the number of 
offspring of a single particle is bounded by~$M$) 
for the random sum $|Z_{n+1}^+(x+1)|$ of $|Z_n(x)|$ 
i.i.d.\ random variables with mean $\po\mo$ yields
\begin{equation} \label{crc2}
\p\Big(|Z_{n+1}^+(x+1)| >(\po\mo+\varepsilon) |Z_n(x)| 
\Big|\sigma(|Z_n(x)|)\Big) \leq \exp \big(-|Z_n(x)| C_1 \big)
\end{equation}
for some constant $C_1>0$ and 
\begin{equation} \label{crc3}
\p\Big(|Z_{n+1}^+(x+1)| < (\po\mo-\varepsilon) |Z_n(x)| 
\Big|\sigma(|Z_n(x)|)\Big) \leq \exp \big(-|Z_n(x)|C_2 \big)
\end{equation}
for some constant $C_2>0$. From~\eqref{crc2} and~\eqref{crc3} we can conclude
\begin{equation} \label{crc4}
\p\Big(D^+_n(x)\Big)\leq \exp(-n C_1) + \exp(-nC_2).
\end{equation}
The same argument leads to analogue estimates for the 
sets $D^-_n(x)$, $E^+_n$ and $E^-_n$ with constants 
$C_i>0$ for $i=3,\ldots,8$. Since at time $n \in \N_0$ particles can only
 be located at the $n+1$ positions $-n, -n+2,\ldots,n-2,n$, we get
\[
\p \left( {\new F_n } \right) \leq 2(2 + 2 (n+1))\exp(-n C)
\]
for $C:=\min\limits_{i=1,\ldots,8} C_i >0$. 
Therefore, the Borel-Cantelli lemma implies~\eqref{crc1}.
\renewcommand{\qedsymbol}{$\square$}
\end{proof}

In the considered case the CBRW behaves very differently 
depending on whether we have $\po\mo\le1$ or $\po\mo>1$:
\begin{itemize}
\item[(i)] For $\po\mo\le1$ the offspring of a single particle 
which move to the right in every step behave as a critical or 
subcritical GWP as long as the particles do not reach the cookies. 
Therefore, we can expect that the amount of particles which reach a 
cookie at the same time is not very large. More precisely, 
we will show in Proposition~\ref{crccor0} that the amount of particles 
in the LP does not grow exponentially.
\item[(ii)] For  $\po\mo>1$ the amount of offspring which moves
 to the right in every time step in the corresponding BRW without 
cookies constitutes a supercritical GWP. Therefore, the number of 
particles at the rightmost occupied position in the 
BRW without cookies a.s.\ grows with exponential rate $\po\mo~>~1$. 
In this case the following proposition shows that the amount 
of particles in the LP is essentially bounded by the growth 
rate of the rightmost occupied position of the corresponding 
BRW without cookies.
\end{itemize}
\begin{prop} 
\label{crccor0}
For every $\alpha > \max\{1, \po\mo\}=:\mm$ we have
\begin{equation} \label{crccor1}
\p\left(\liminf_{n \to \infty} \{|\lp{n}| < \alpha^n\} \right)=1.
\end{equation}
\end{prop}
\begin{proof}[Proof of Proposition~\ref{crccor0}]
For the proof we start with the following lemma which states that a 
large LP at time $n$ leads to a long survival of the LP afterwards 
(except for finitely many times). For $\beta > 0$ we define
\begin{align}
& G_n:= G_n(\beta):=  \left\{|\lp{n}|\ge n,\, \tau(n)\le \beta  \log 
|\lp{n}|\right\}, \label{crc5a}
\intertext{where}
& \tau(n):=  \inf\{\ell\geq n:\, |\lp{\ell}|=0\} \label{crcdeftaun}
\end{align}
denotes the time of the next extinction of the LP beginning from time $n$.
\begin{lem} 
\label{crclem0.2}
There exists  some $\beta > 0$ such {\new that} we have
\begin{equation} \label{crc6}
\p\left( \limsup_{n \to\infty} G_n \right) = 0.
\end{equation}
\end{lem}
\begin{proof}[Proof of Lemma \ref{crclem0.2}]
Let us first look at a subcritical GWP 
$\big(GW_n^{\textnormal{sub}}\big)_{n \in \N_0}$ 
with reproduction mean $\pc\mc<1$ and strictly positive,
finite offspring variance and its extinction time $T^{\textnormal{sub}}$.
Assuming that we have an initial population of $z \in \N$ particles, 
we get using Proposition~\ref{propo1}
\[ \p_z(T^{\textnormal{sub}}\le n) =\big(1-\p(GW_n^{\textnormal{sub}}>0)\big)^z \leq \big(1 - c(\pc\mc)^n\big)^z \leq {\new \exp\big(- c(\pc\mc)^n z \big)},\] 
{\new for a suitable constant $c>0$. In particular, if the LP is subcritical, we conclude that}
\begin{align*}
& {\new  \p\big(|\lp{n}|\ge n,\, \tau(n)\le \beta  \log |\lp{n}|\, \big|\, |\lp{n}| \big) \le 
\exp\left(- c(\pc\mc)^{\beta  \log |\lp{n}|} |\lp{n}|  \right) \cdot \mathds{1}_{\{|\lp{n}|\ge n\}}} \\
{\new \le}\ & {\new  \exp\left(- c \cdot n^{\frac12}  \right)}
\end{align*}
{\new for $0 < \beta$ small enough and all $n \in \N$. Therefore, the Borel-Cantelli lemma implies~\eqref{crc6}. In the case of a critical LP, we can use an analogous argument together with Proposition~\ref{propo3}.}
\renewcommand{\qedsymbol}{$\square$}
\end{proof}

In the following we want to investigate the behaviour of the CBRW on the event
\begin{equation} \label{crcdefhno}
H_{n_0}:=\bigcap\limits_{n \geq n_0} \big(F_n^c\cap G_n^c\big)
\end{equation} 
for fixed $n_0 \in \N_0$. {\new (We will later choose $n_0$ large enough such that the assumptions of the upcoming Lemma~\ref{crclem17} and equation~\eqref{crcpropgam} are satisfied.)} On this event we have upper and lower bounds for 
\[
\frac{|Z_{n+1}^+(x+1)|}{|Z_n(x)|} \quad \text{and} 
   \quad \frac{|Z_{n+1}^-(x-1)|}{|Z_n(x)|}
\]
for positions $x \in \Z$ containing at least~$n$
 particles at time $n\ge n_0$ (cf.~\eqref{crc1a}). 
Additionally, we have a lower bound for the time for which a LP with at 
least~$n$ particles at time $n \ge n_0$ will stay alive afterwards 
(cf.~\eqref{crc5a}). We note that we have
\begin{equation} \label{crcprophn}
 \lim_{n \to \infty} \p(H_n) =  \p\left( \liminf_{n \to \infty} \big(F_n^c\cap G_n^c\big) \right)
=1
\end{equation}
due to Lemma \ref{crclem1} and Lemma \ref{crclem0.2}.
\\[10pt]
For the next lemma we need some additional notation. We define
\[
\sigma_0:= \inf \{n > n_0:\, |\lp{n-1}|=0,\, |\lp{n}| 
\neq 0,\, l(n)\leq n - 2n_0 - 1\},
\]
which is the time of the first rebirth of the LP after time~$n_0$ for 
which we have
\[
l(\sigma_0) - (\sigma_0 - n_0) \le - 2n_0 - 1 + n_0 = -n_0 - 1.
\]
This implies
\begin{equation} \label{crc7.1}
{\new \left|Z_{n_0}\big(l(\sigma_0) - (\sigma_0 - n_0 + k)\big)\right|}=0,
\end{equation}
{\new for all $k \in \N_0$} which is an important fact which we make use of in the 
following calculations (cf.\ Figure~\ref{pic1}).
Since the LP is critical or subcritical and the BRW without 
cookies is transient to the right, we a.s.\ have $\sigma_0< \infty$. 
We now define the random times 
\begin{align*}
\tau_n&:=\inf \{\ell >\sigma_{n}:\ |\lp{\ell}|=0 \}-\sigma_n ,
~\text{for }n \ge0,\\
\sigma_n&:= \inf \{\ell > \sigma_{n-1} + \tau_{n-1}:\ |\lp{\ell}| \neq 0 \} ,
~\text{for }n \ge 1,
\end{align*}
which denote the time period of survival and the time of {\new the restart}
of the LP, inductively. 
Due to the assumptions of the CBRW all of these random times are a.s.\ finite. 
Using  \eqref{crc7.1} we see that we have
\begin{equation}
\label{crc7.1a}
{\new \left|Z_{n_0}\big(l(\sigma_j) - (\sigma_j - n_0+k)\big)\right|}=0
\end{equation}
for all $j,k\in\N_0$. {\new We note that the argument in \eqref{crc7.1} is true for all $n \le n_0$ instead of $n_0$. Therefore we can conclude that}
\begin{equation}
\label{crc7.1a.x}
{\new {\new \left|Z_{n}\big(l(\sigma_j) - (\sigma_j - n+k)\big)\right|}=0}
\end{equation}
{\new holds for all $n \le n_0$ and for all $j,k\in\N_0$.}

As the next step of the proof, we state the following upper bounds 
for the size of the {\new LP on the event $H_{n_0}$}:
\begin{lem} \label{crclem17}
{\new For the particles in the LP we have the following upper bound for $k, n, z \in \N$, $n \ge n_0$:
\begin{alignat}{2}
 |\lp{n+k}| \le  z  (\pc\mc+\varepsilon)^k + k  M  (n+k-1)
 (1+\delta)^{k-1}~\text{on }H_{n_0} \cap \{|\lp{n}|=z\} \cap\{\tau(n)\ge n+k\}.
\label{crccalc6}
\end{alignat}}
{\new Further, for arbitrary $\gamma>0$, there exists $n^*=n^*(\gamma)$ such that}
\begin{align}
&|\lp{\sigma_{j+1}}| \le\, (\mm+3\gamma)^{\sigma_{j+1}}
&&\text{on } H_{n_0} \cap \{\sigma_{j+1} = \sigma_j + \tau_j + 2\}
\cap \{|\lp{\sigma_j}| \le (\mm+\gamma)^{\sigma_{j}}\}, \label{crcest1}\\
&|\lp{\sigma_{j+1}}| \le\, |\lp{\sigma_{j}}|  (\mm+4\gamma)^{\tau_j} 
&&\text{on } H_{n_0} \cap \{\sigma_{j+1} = \sigma_j + \tau_j + 2\} 
\cap \{|\lp{\sigma_j}| > (\mm+\gamma)^{\sigma_{j}}\}, \label{crcest2}\\\
&|\lp{\sigma_{j+1}}| \le\, (\mm+2\gamma)^{\sigma_{j+1}}&&\text{on } H_{n_0} \cap \{\sigma_{j+1} > \sigma_j + \tau_j + 2\}, \label{crcest3}
\end{align}
{\new holds for all $j \in \N_0$ and $n_0 \ge n^*$}, where $\mm=\max\{1,\po\mo\}$.
\end{lem}

\begin{proof}[Proof of Lemma~\ref{crclem17}]
First we choose $0<\delta<\gamma$ in such a way that
\begin{align}
1+\delta \le \,  \frac{\mm + 2\gamma}{\mm+\gamma}, \quad
1+\delta \le \,  \left(\frac{\mm + 3\gamma}{\mm+2\gamma}\right)^{\beta \log(\mm+\gamma)}, \label{crcdelta}
\end{align}
where $\beta>0$ satisfies Lemma \ref{crclem0.2}. 
Then we choose $\varepsilon>0$ for the definitions of the sets 
$\big(F_n\big)_{n \in \N_0}$ (cf.\ \eqref{crc1a}) sufficiently small  such that
\begin{align}
\hspace{1pt}\pc\mc+\varepsilon \le \,   1+\delta, \quad
{\new 1 < }\, \frac{\po\mo+\varepsilon}{\po\mo-\varepsilon}\le \,  1 +\delta, \quad
\po\mo+\varepsilon \le \,  \mm + \gamma. \label{crc7} 
\end{align}
For the upcoming estimates we use the following properties of the 
set $H_{n_0}$. For $n > n_0$ we have
\begin{equation}
\label{crccalc1}
|Z_{n-1}(x-1)| \le n-1
~\text{on }H_{n_0} \cap \{|Z_n(x)|=0\},
\end{equation}
which means that there cannot be very many 
particles at position $x-1$ one time step before $n$ if we know that 
the position $x$ stays empty at time~$n$. 
Similarly, the knowledge of $|Z_n(x)|$ gives us upper estimates 
for $(|Z_{n-k}(x-k)|)_{k \in \N}$.
 If we are in the case in which the cookies are always to the right of 
the considered positions, we have for $n>n_0$
\begin{alignat}{2}
&|Z_{n-1}(x-1)| \le  z  (\po\mo - \varepsilon)^{-1} + n-1
\qquad \text{on }H_{n_0} \cap \{|Z_n(x)|=z,\ l(n-1) > (x-1)\},\nonumber\\
&{\new |Z_{n-k}(x-k)| \le  z  (\po\mo - \varepsilon)^{-k} + (n-1)
 \big(1 \vee (\po\mo-\varepsilon)^{-k+1} \big)} \nonumber\\
& \hphantom{|Z_{n-1}(x-1)| \le  z  (\po\mo - \varepsilon)^{-1} + n-1} \qquad~\text{on }H_{n_0} \cap \{|Z_n(x)|=z,\ l(n-1) > (x-1)\}\label{crccalc3}
\end{alignat}
for $n-k \ge n_0$. The first estimate is easily obtained using a proof by contradiction and an iteration of it yields the second inequality. {\new We note here that by construction the upper bound is at least equal to $n_0$. Therefore, if the upper bound is exceeded, at least a ratio of $(\po\mo - \varepsilon)$ of $|Z_{n-k}(x-k)|$ will contribute to $|Z_{n-k+1}(x-k+1)|$ on the considered event due to the definition of $H_{n_0}$. This gives the contradiction.}\\
{\new For $n \ge n_0$ and $k\in \N$}, we obtain similar estimates for the size of the LP before the next extinction at time $\tau(n)$ (for the definition of $\tau(n)$ 
see~\eqref{crcdeftaun}):
\begin{alignat}{2}
&|\lp{n+1}| \le  z  (\pc\mc+\varepsilon) + M  n 
 ~\text{on }H_{n_0} \cap \{{\new|\lp{n}|}=z\},\nonumber\\
&|\lp{n+2}| \le  z  (\pc\mc+\varepsilon)^2 + 2  M  (n+1) 
 {\new \big(1 \vee (\pc\mc+\varepsilon)\big)} ~\text{on }H_{n_0} \cap \{|\lp{n}|=z\} \cap\{\tau(n)\ge n+2\},
\nonumber\\
 &|\lp{n+k}| \le  z  (\pc\mc+\varepsilon)^k + k  M  (n+k-1)
 (1+\delta)^{k-1}~\text{on }H_{n_0} \cap \{|\lp{n}|=z\} \cap\{\tau(n)\ge n+k\}.
\nonumber
\end{alignat}
{\new For the last upper bounds, we note that we can distinguish between the following two cases: If $|\lp{n+k}| \le n + k - 1$, then we have $|\lp{n+k+1}| \le M(n + k - 1)$ due to assumption \eqref{assm}. Otherwise we can use the definition of $H_{n_0}$ to get $|\lp{n+k+1}| \le (\pc\mc+\varepsilon) |\lp{n+k}|$ on the considered set. In particular, we have shown \eqref{crccalc6}.}\\
Now, we introduce two processes $(\Phi_n)_{n\in\N}$ and $(\Psi_n)_{n\in\N}$,
 which help us -- together with the estimates~{\new \eqref{crccalc6}, \eqref{crccalc1}, and
\eqref{crccalc3}} -- 
to control the number of particles that restart the 
LP at time $\sigma_{j+1}$ (cf.\ Figure~\ref{pic1} and \ref{pic2}). 
For $j\in\N_0$ and $n\in\N$ we define
\[
	 \Phi_n^{(j)}  := Z_{n}(l(\sigma_{j+1})-\sigma_{j+1}+n) 
~\mbox{and}~
	\Psi_n^{(j)}  :=Z_{n}(l(\sigma_{j+1})-\sigma_{j+1}+2+n).
\]
For sake of a better presentation we drop the superscript~$j$ and write just $\Phi_{n}$ and $\Psi_{n}$ if there is no room for confusion.
We observe that we have
\begin{align}
|\Phi_{n+1}| &=|\Psi_{n}|=0\label{crc31}
\intertext{for all $n\le n_0$ due to~{\new \eqref{crc7.1a.x}.}
Furthermore, by definition we have 
$\Phi_{\sigma_{j+1}}=\lp{\sigma_{j+1}}$ and}
 |\Psi_{\sigma_{j+1}}| &=|\Psi_{\sigma_{j+1}-1}|
=|\Psi_{\sigma_{j+1}-2}|=0.\label{crc7.1b}
\end{align}
Again, we split the set of particles $\Phi_n$ into the particles which have moved one step to the right from time~$n-1$ to time~$n$ and the particles which have moved to the left:
\begin{align*}
\Phi_n^+ := &Z_{n}^+(l(\sigma_{j+1})-\sigma_{j+1}+n), \\
\Phi_n^- := &Z_{n}^-(l(\sigma_{j+1})-\sigma_{j+1}+n).
\end{align*}
\begin{figure}[ht]
\vspace{45pt} \hspace{2cm}
\includegraphics [viewport=50 140 210 450, scale=0.75]{pic1}
  \caption{The LP is restarted at time $\sigma_{j+1}$ 
two time steps after the last extinction at time $\sigma_j+\tau_j$. 
The two diagonals represent the processes $(\Phi_n)_{n\in\N}$ 
and $(\Psi_n)_{n\in\N}$.} \label{pic1}
\end{figure}
To obtain an upper bound for {\new$|\Phi_{\sigma_{j+1}}| =  |\lp{\sigma_{j+1}}|$}, 
we use the following recursive structure. We have 
$|\Phi_n^-| \le\, M  |\Psi_{n-1}| $
for $n\in\N$ due to assumption~\eqref{assm}. 
Moreover, on $H_{n_0}$ we have 
$|\Phi_n^+| \le\, |\Phi_{n-1}|  (\po\mo + \varepsilon) 
+ M  \sigma_{j+1}$
for $n_0 +2\le n \le \sigma_{j+1}$ 
(since the particles reproduce and move without cookies) 
and these two facts yield
\begin{align}
|\Phi_n| &=\, |\Phi_n^+| + |\Phi_n^-|
\le |\Phi_{n-1}|  (\po\mo + \varepsilon) + M 
 \sigma_{j+1}+M  |\Psi_{n-1}| \label{crc30}
\end{align}
for $n_0 + 2\le n \le \sigma_{j+1}$. Using~\eqref{crc31}, 
\eqref{crc7.1b}, and $\sigma_{j+1} - n_0 -1$ iterations of
the recursion in~\eqref{crc30},
 we obtain the following upper bound for the particles which start the LP 
at time $\sigma_{j+1}$ on $H_{n_0}$:
\begin{align}
 |\Phi_{\sigma_{j+1}}| \le\, & M\sum_{k=3}^{\sigma_{j+1} - n_0-1} 
|\Psi_{\sigma_{j+1}-k}| (\po\mo + \varepsilon)^{k-1}
+M\sigma_{j+1}\sum_{k=1}^{\sigma_{j+1} - n_0-1}  
(\po\mo + \varepsilon)^{k-1} \allowdisplaybreaks[1]\nonumber\\ 
 \le\, & M\sum_{k=1}^{\sigma_{j+1} - n_0-3} 
|\Psi_{\sigma_{j+1}-k-2}| (\po\mo+ \varepsilon)^{k+1}
+ M \sigma_{j+1}^2 (\mm + \gamma)^{\sigma_{j+1}} .
\label{crc32} 
\end{align}
Note that this bound just depends on $\sigma_{j+1}$ and 
the process $(\Psi_n)_{n\in\N}$. 
For this reason we now take a closer look at $(\Psi_n)_{n\in\N}$ and 
distinguish between the following two cases:
\begin{itemize}
\item In the first case we assume that the LP restarts right after it has died out and we therefore have $\sigma_{j+1} = \sigma_j + \tau_j + 2$. In this case the process $(\Psi_n)_{n\in\N}$ coincides with the LP between time $\sigma_j$ and $\sigma_j+\tau_j$ (cf. Figure~\ref{pic1}).
\item
In the second case we assume that we have $\sigma_{j+1} > \sigma_j + \tau_j + 2$. From this we know that there are no particles in the LP at time $\sigma_{j+1}-2$ and thus the process $(\Psi_n)_{n\in\N}$ is always left of the cookies (cf.\ Figure~\ref{pic2}).
\end{itemize}
In both cases the crucial observation is that the amount of particles in $(\Psi_n)_{n\in\N}$ does not exceed a certain level since none of its offspring reaches the leftmost cookie at time $\sigma_{j+1}-2$.

First, we consider the case $H_{n_0} \cap \{\sigma_{j+1} = \sigma_j + \tau_j + 2\}$. We apply the estimations~{\new \eqref{crccalc6} and \eqref{crccalc3}} to give upper bounds for $|\Psi_{\sigma_{j+1}-k}|=|\Psi_{\sigma_{j} + \tau_j+2-k}|$ for $1\le k\le\sigma_{j+1}-n_0$. We know by definition of $\sigma_j$ that we have $l(\sigma_j-1)=l(\sigma_j)>l(\sigma_j)-1$. Thus, we can apply~\eqref{crccalc3} and conclude that on the event $H_{n_0}$ for $1 \le k \le \sigma_{j}-n_0$ we have
\begin{align*}
|\Psi_{\sigma_{j} - k}| =&\, |Z_{\sigma_j-k}(l(\sigma_j)-k)|
\le |\lp{\sigma_j}|  (\po\mo-\varepsilon)^{-k} + \sigma_j 
{\new \big(1 \vee (\po\mo-\varepsilon)^{-k+1}\big)}
\end{align*}
and by using \eqref{crccalc6} for $0 \leq k \leq \tau_j -1$ 
we get
\begin{align*}
|\Psi_{\sigma_j + k}| =&\, |\lp{\sigma_j + k}| \le  |\lp{\sigma_j}| 
  (\pc\mc + \varepsilon)^k  
+  k  M (\sigma_j + k - 1)
(1+\delta)^{k-1}. 
\end{align*}
Applying these two estimates to~\eqref{crc32} yields
 \begin{align}
 |\Phi_{\sigma_{j+1}}| \leq \ & M \sum_{k=1}^{\tau_j} 
|\Psi_{\sigma_{j}+(\tau_j-k)}| (\po\mo+ \varepsilon)^{k+1} 
 + M \sum_{k=\tau_j+1}^{\sigma_{j}+\tau_j-n_0-1}
|\Psi_{\sigma_{j}-(k-\tau_j)}| (\po\mo+ \varepsilon)^{k+1} \nonumber\\
  & + M \sigma_{j+1}^2  (\mm + \gamma)^{\sigma_{j+1}} 
 \nonumber \\
  \leq \ & M \sum_{k=1}^{\tau_j}  \Big(|\lp{\sigma_{j}}|
 (\pc\mc+\varepsilon)^{\tau_j-k}
+ (\tau_j-k)  M 
 (\sigma_{j}+\tau_j-k-1)  (1+\delta)^{\tau_j-k-1}\Big) 
(\po\mo+ \varepsilon)^{k+1} 
\nonumber \\
  & + M \sum_{k=\tau_j+1}^{\sigma_{j}+\tau_j-n_0-1}  
\Big(|\lp{\sigma_j}|  (\po\mo -\varepsilon)^{-k +\tau_j}
+ \sigma_j 
 {\new \big(1 \vee (\po\mo -\varepsilon)^{-k +\tau_j+1}\big)} \Big) 
 (\po\mo+ \varepsilon)^{k+1} 
\nonumber \\
  & + M \sigma_{j+1}^2 (\mm+\gamma)^{\sigma_{j+1}}
 \nonumber \\
  \le \ & \tau_j^2 M^2 (|\lp{\sigma_{j}}| + \sigma_{j+1}) 
 (1+\delta)^{\tau_j-1} (\mm+\gamma)^{\tau_j+1}\nonumber \\
  &  + {\new \sigma_{j+1} M |\lp{\sigma_{j}}| 
\left(\frac{\po\mo+\varepsilon}{\po\mo-\varepsilon}\right)^{\sigma_{j}}
{\new(\po\mo+\varepsilon)^{\tau_j+1} }
+  2 M  \sigma_{j+1}^2 (\mm+\gamma)^{\sigma_{j+1}}}
\nonumber \\
  \le \ &  2 M^2 \sigma_{j+1}^2 
(|\lp{\sigma_j}|+\sigma_{j+1}) 
(1+\delta)^{\sigma_j + \tau_j - 1} {\new(\mm+\gamma)^{\tau_j+1} }
+ {\new 2}M  \sigma_{j+1}^2 (\mm+\gamma)^{\sigma_{j+1}}.
\label{crcest1a} 
 \end{align}
Here we use~\eqref{crc7} in the last two steps.

If we first investigate $|\lp{\sigma_{j+1}}|$ on the subset 
$\{|\lp{\sigma_j}| \le (\mm + \gamma)^{\sigma_{j}}\} \cap H_{n_0} 
\cap \{\sigma_{j+1} = \sigma_j + \tau_j + 2\}$,
on which we have a limited amount of particles in $\lp{\sigma_j}$, 
we get by using~\eqref{crcest1a}
\begin{align*} 
|\lp{\sigma_{j+1}}| = \ & |\Phi_{\sigma_{j+1}}| \\
\le \ &   2 M^2 \sigma_{j+1}^2 
\big((\mm+\gamma)^{\sigma_{j}}+\sigma_{j+1}\big) 
(1+\delta)^{\sigma_j + \tau_j - 1} {\new (\mm+\gamma)^{\tau_j+1}}
+{\new 2} M  \sigma_{j+1}^2 (\mm+\gamma)^{\sigma_{j+1}}
\nonumber \\
\le \ &  {\new 4} M^2  (\sigma_{j+1}+1)^3 
 (1+\delta)^{\sigma_{j+1}}  (\mm+\gamma)^{\sigma_{j+1}}
\nonumber \\
\le \ & (\mm+3\gamma)^{\sigma_{j+1}} 
\end{align*}
for $n_0$ and thus $\sigma_{j+1} \ge n_0$ large enough 
due to~\eqref{crcdelta}. This shows~\eqref{crcest1} 
in Lemma~\ref{crclem17}.

On the other hand, if we consider the remaining subset
$\{|\lp{\sigma_j}| > (\mm+\gamma)^{\sigma_{j}}\} \cap H_{n_0} 
\cap \{\sigma_{j+1} = \sigma_j + \tau_j + 2\}$,
\eqref{crcest1a} yields
\begin{align*} 
|\lp{\sigma_j}|^{-1} |\lp{\sigma_{j+1}}| = \ & |\lp{\sigma_j}|^{-1} 
 |\Phi_{\sigma_{j+1}}|\\
 \le \ & 2  M^2 \sigma_{j+1}^2 (1+\sigma_{j+1}) 
(1+\delta)^{\sigma_j + \tau_j - 1} (\mm+\gamma)^{\tau_j+1} 
+ {\new 2} M \sigma_{j+1}^2 (\mm+\gamma)^{\tau_{j}+2} 
\nonumber \\
\le \ & {\new 4}  M^2 (\sigma_j+\tau_j+3)^3 
 (1+\delta)^{\sigma_j} (\mm+2\gamma)^{\tau_{j}+2} 
\nonumber \\
\le \ & {\new 4}  M^2  (\sigma_j+\tau_j+3)^3 
  (1+\delta)^{\frac{1}{\beta \log(\mm+\gamma)} \tau_j} 
(\mm+2\gamma)^{\tau_{j}+2} \nonumber \\
\le \ & (\mm+4\gamma)^{\tau_{j}} 
\end{align*}
for $n_0$ and thus $\sigma_{j} \ge n_0$ large enough. 
Here we use~\eqref{crcdelta} and the fact that we have
 $\{\tau_j > \beta \log \big((\mm+\gamma)^{\sigma_j}\big)\}$ 
on the considered event (cf.\ Lemma~\ref{crclem0.2}). 
This shows~\eqref{crcest2} in Lemma~\ref{crclem17}.

\begin{figure}\hspace{2cm}
\includegraphics [viewport=60 135 210 400, scale=0.75]{pic2}
 \caption{The LP is restarted at time $\sigma_{j+1}$ 
more than two time steps after the last extinction at time $\sigma_j+\tau_j$. The two diagonals represent the processes $(\Phi_n)_{n\in\N}$ and $(\Psi_n)_{n\in\N}$.} \label{pic2}
\end{figure}

We now consider the event $H_{n_0} \cap \{\sigma_{j+1} > \sigma_j + \tau_j + 2\}$. First, we observe that on this set, due to~\eqref{crccalc1}, we have
\begin{align}
|\Psi_{\sigma_{j+1}-2-1}| =\ & |Z_{\sigma_{j+1}-2-1}(l(\sigma_{j+1}) - 1)| 
\le\  \sigma_{j+1}-2-1 \le \sigma_{j+1} \label{crcest17.1}
\intertext{since $|\Psi_{\sigma_{j+1}-2}|
=|Z_{\sigma_{j+1}-2}(l(\sigma_{j+1}))|=0$ holds. 
Further, we observe that the particles which belong to 
$(\Psi_n)_{n\in\N}$ are always to the
 left of the cookies. In particular, we have
$l(\sigma_{j+1}-2-1)=l(\sigma_{j+1})>l(\sigma_{j+1})-1$.
Therefore, we can apply \eqref{crccalc3} and 
conclude, by using \eqref{crcest17.1},}
|\Psi_{\sigma_{j+1}-2-k}| =\ & |Z_{\sigma_{j+1}-2-k}(l(\sigma_{j+1}) - k)| 
\nonumber\\
 \le\ &\sigma_{j+1} (\po\mo - \varepsilon)^{-k} + (\sigma_{j+1}-2-1) 
{\new \big(1 \vee (\po\mo - \varepsilon)^{-k+1}\big)} \nonumber \\
\le \ &  2\sigma_{j+1} {\new \big(1 \vee (\po\mo - \varepsilon)^{-k} \big)}
\label{crccalc101}
\end{align}
for $2 \le k \le \sigma_{j+1} - 2 - n_0$.

With the help of~\eqref{crc32} and~\eqref{crccalc101} 
we get on the event $H_{n_0} \cap \{\sigma_{j+1} > \sigma_j + \tau_j + 2\}$
\begin{align*} 
|\Phi_{\sigma_{j+1}}| \le \ & M \sum_{k=1}^{\sigma_{j+1}-n_0-3}  
|\Psi_{\sigma_{j+1}-2-k}| (\po\mo+ \varepsilon)^{k+1} 
+ M \sigma_{j+1}^2  (\mm+\gamma)^{\sigma_{j+1}} 
\nonumber \\
\le \ &  M \sum_{k=1}^{\sigma_{j+1}-n_0-3}  2 \sigma_{j+1} 
 {\new \big(1 \vee (\po\mo - \varepsilon)^{-k}\big)} (\po\mo+\varepsilon)^{k+1} 
+ M \sigma_{j+1}^2 (\mm+\gamma)^{\sigma_{j+1}} 
\nonumber \\
\le \ &  2  M \sigma_{j+1}^2 
{\new \left( (\po\mo+\varepsilon)^{\sigma_{j+1}-n_0-3} \vee\left(\frac{\po\mo+\varepsilon}{\po\mo-\varepsilon}
\right)^{\sigma_{j+1}-n_0-3} \right)}  (\po\mo+\varepsilon)
 + M  \sigma_{j+1}^2 (\mm+\gamma)^{\sigma_{j+1}}
\nonumber\\
\le \ & 3 M \sigma_{j+1}^2 (\mm+\gamma)^{\sigma_{j+1}} 
\nonumber \\
\le \ & (\mm+2\gamma)^{\sigma_{j+1}} 
\end{align*}
for $n_0$ and thus $\sigma_{j+1} \ge n_0$ large enough. 
Here we use~\eqref{crcdelta} and~\eqref{crc7} in the last two steps. 
This shows~\eqref{crcest3} in Lemma~\ref{crclem17}.
\renewcommand{\qedsymbol}{$\square$}
\end{proof}
We now return to the proof of Proposition~\ref{crccor0}. 
First, we choose $\gamma \in \R$ with
$0< 6\gamma < \alpha-\mm$
and~$n_0$ large enough such that the estimations~\eqref{crcest1}, 
\eqref{crcest2} and \eqref{crcest3} from Lem\-ma~\ref{crclem17} hold. 
Using these estimations, we can conclude that on $H_{n_0}$ we a.s.\ have
\begin{equation} 
\label{crcetafin}
\eta:= \inf\{n \geq n_0:\, |\lp{\sigma_n}| < (\mm+5\gamma)^{\sigma_{n}}\}
< \infty.
\end{equation}
To see this, we just have to see what happens on the event 
\[
H_{n_0} \cap \bigcap\limits_{j=1}^k \Big(\big\{|\lp{\sigma_j}| 
> (\mm+\gamma)^{\sigma_{j}}\big\} \cap \big\{\sigma_{j+1} 
= \sigma_j + \tau_j + 2\big\} \Big).
\]
On this event, we can use~\eqref{crcest2} $k$ times in a row and we get
\[
|\lp{\sigma_{k}}| \leq |\lp{\sigma_0}|\prod_{j=1}^{k} 
(\mm+4\gamma)^{\tau_j} \leq  |\lp{\sigma_0}|  (\mm+4\gamma)^{\sigma_k},
\]
from which we conclude that \eqref{crcetafin} indeed holds on~$H_{n_0}$.

Again by using the three estimations~\eqref{crcest1}, \eqref{crcest2},
 and~\eqref{crcest3} of Lemma~\ref{crclem17}, 
we can see inductively that on the event~$H_{n_0}$ we have
$|\lp{\sigma_{n}}| \leq (\mm+5\gamma)^{\sigma_n}$
for all $n \geq \eta$. Additionally, if we assume
$|\lp{\sigma_n + i-1}|\le (\mm+5\gamma)^{\sigma_n+i-1}$,
we see inductively by using~\eqref{crccalc6} that on the event $H_{n_0}$ 
we have for all $n \geq \eta$ and for all $1\leq i \leq \tau_n-1$ 
\begin{align*}
 |\lp{\sigma_n + i}|\leq \ & |\lp{\sigma_n + i-1}| (\pc\mc+\varepsilon) 
+ (\sigma_n+i-1)  M \\
\leq\ & (\mm+5\gamma)^{\sigma_n+i-1}  (\pc\mc+\varepsilon) 
+ (\sigma_n+i-1)  M \\
\le \ & (\mm+5\gamma)^{\sigma_n+i-1}  (\mm+\gamma) + (\sigma_n+i-1) 
 M \\
\leq\ & (\mm+6\gamma)^{\sigma_n+i} < \alpha^{\sigma_n+i}
\end{align*}
for $n_0$ (and thus also $\sigma_n\geq n_0$) large enough. 
Since by definition of $(\sigma_n)_{n\in\N_0}$ and $(\tau_n)_{n\in\N_0}$ 
the LP is empty at the remaining times, we conclude that we have
\begin{equation}\label{crccalc201}
\p\left(\liminf_{n\to\infty}\big(H_n\cap\{|\lp{n}|<\alpha^n\}\big) \right)=1.
\end{equation}
Finally, {\new \eqref{crccalc201} together with
$\p\left(\liminf_{n\to\infty}H_n\right)=1$ (cf.\ \eqref{crcprophn}),
yields \eqref{crccor1} and finishes the proof of Proposition \ref{crccor0}.}
\renewcommand{\qedsymbol}{$\square$}
\end{proof}
After having investigated the growth of the LP, 
we are now interested in the speed at which the cookies are consumed:
\begin{prop} 
\label{crcprop2}
\renewcommand{\labelenumi}{(\alph{enumi})}
\renewcommand{\labelenumii}{(\roman{enumii})}
\begin{enumerate}
\item There exists $\lambda > 0$ such that we a.s.\ have
   \begin{equation} \label{crclem5.1}
\liminf_{n \to \infty} \frac{l(n)}{n}> \lambda.
   \end{equation}
\item In fact, for $\po\mo>1$ we a.s.\ have
  \begin{equation} \label{crclem5.2}
\lim_{n \to \infty} \frac{l(n)}{n} = 1.
\end{equation}
\end{enumerate}
\end{prop}
\begin{proof}[Proof of Proposition~\ref{crcprop2}]
(a) We compare the CBRW with the following 
process $(W_n)_{n \in \N_0}$, that behaves similarly to an excited 
random walk. It  is determined by the initial configuration $W_0:=0$ and the transition probabilities
\[
 \p\big(W_{n+1}=W_n+1\mid (W_j)_{1\le j\le n}\big)=\begin{cases} 0 
&\text{on } \left\{\max\limits_{j=0,1,\ldots,n-1}W_j<W_n\right\} 
\vspace{2pt}\\
 \po &\text{on } \left\{\max\limits_{j=0,1,\ldots,n} W_j>W_n\right\} 
\end{cases}
\] 
and
\[
\p\big(W_{n+1}=W_n-1\mid (W_j)_{1\le j\le n}\big)=
\begin{cases} 1 
&\text{on } \left\{\max\limits_{j=0,1,\ldots,n-1}W_j<W_n\right\} \vspace{2pt} \\
\qo &\text{on } \left\{\max\limits_{j=0,1,\ldots,n} W_j>W_n\right\} 
\end{cases}  
\]
for $n \in \N_0$.
The process $(W_n)_{n \in \N_0}$ moves to the left with probability~$1$
 every time it reaches a position $x \in \N_0$ for the first time and 
otherwise it behaves as an asymmetric random walk on $\Z$ with transition 
probabilities~$\po$ and~$\qo$. For the random times
$T_x:= \inf\{n \in \N_0:\, W_n=x\}$ (for $x\in\N_0$),
we notice that $\big(T_{x+1}-T_{x}\big)_{x \in \N_0}$ is a sequence of 
i.i.d.\ random variables with
\[
\E[T_{1}-T_{0}]= \E[T_{1}] = 1 + \frac{{\new 2}}{2\po-1}.
\]
Therefore, the strong law of large numbers implies that we a.s.\ have
\[
\lim_{x \to \infty} \frac{T_{x}}{x} = \lim_{x \to \infty} 
\frac1{x}{\sum_{i=0}^{x-1} (T_{i+1}-T_{i})} 
= \E[T_{1}-T_{0}] = 1 + \frac{{\new 2}}{2\po-1}<\infty.
\]
Since we can couple the CBRW and the process $(W_n)_{n \in \N_0}$ 
in a natural way such that we have 
$\max_{\nu \in Z_n} X_{\nu} \ge W_n$
for all $n \in \N_0$, we can conclude that~\eqref{crclem5.1} holds for 
$0<\lambda<\left(1 + \frac{{\new 2}}{2\po-1}\right)^{-1}$. 
\medskip
\noindent
(b) We start this part of the proof with the following lemma:
\begin{lem} 
\label{crcrem1}
For a CBRW with $\mo>1$, there exists $\gamma>1$ such that we a.s.\ have
\begin{equation} 
\label{crcrem1.1}
\lim_{n \to \infty} \frac{|Z_n|}{\gamma^n} = \infty.
\end{equation}
\end{lem}
\begin{proof}[Proof of Lemma \ref{crcrem1}]
Let us treat the case where $\mc>1$ first. 
Let $\big(V_{n,k}\big)_{n,k\in \N}$ be i.i.d.\ random variables with
\begin{align*}
1-\p(V_{1,1}=1) = \p(V_{1,1}=2)&=\min\left\{\sum_{i=2}^{\infty} \mu_0(i), \sum_{i=2}^{\infty} \mu_c(i)\right\}, 
\end{align*}
and we define the corresponding GWP 
$\big(\widetilde{Z}_n\big)_{n \in \N_0}$ by $\widetilde{Z}_0:=1$,
$\widetilde{Z}_{n+1}:=\sum_{i=1}^{\widetilde{Z}_n} V_{n+1,i}$. Observe, that  $\E[V_{1,1}] > 1$. A standard coupling argument reveals that $\widetilde{Z}_{n}\leq |Z_{n}|$. Now, the claim follows since $\widetilde{Z}_{n}$ grows exponentially, e.g. Theorem~I.10.3 on page~30 in~\cite{Athreya/Ney}.

The other case is similar: 
Consider now  the i.i.d.\ random variables $\big(V_{n,k}\big)_{n,k\in \N}$  with
\begin{align*}
1-\p(V_{1,1}=1) = \p(V_{1,1}=2)&=\min\{\qo,\qc\} \sum_{i=2}^{\infty} \mu_0(i), 
\end{align*}
and  define as above  the corresponding GWP 
$\big(\widetilde{Z}_n\big)_{n \in \N_0}$.  For the coupling we observe that the probability of every particle in the CBRW to produce 
a particle which moves to the left is bounded from below by $\min\{\qo,\qc\}$. 
Such a particle cannot be at a position with a cookie and therefore its 
offspring distribution is given by $\big(\mu_0(i)\big)_{i \in \N_0}$. 
Eventually, the corresponding coupling yields $ \widetilde{Z}_n \le |Z_{2n}|$ and the claim follows as above.
\renewcommand{\qedsymbol}{$\square$}
\end{proof}

We now return to the proof of Proposition~\ref{crcprop2}(b).
Let us choose $\varepsilon>0$ such  that
\begin{equation} 
\label{crcchoicee2}
\po\mo-\varepsilon > 1,\ \new{\qc\mc- \varepsilon >0}.
\end{equation}
We use this $\varepsilon$ for the definition of the sets
 $\big(F_n\big)_{n \in \N_0}$ and $\big(H_n\big)_{n \in \N_0}$, see \ \eqref{crc1a} and \ \eqref{crcdefhno}. 
Due to Lemma \ref{crcrem1}  we can choose~$\gamma>1$ such  that we a.s.\ have
\begin{align}
\lim_{n \to \infty} \frac{|Z_n|}{\gamma^{2n}}=\infty & \quad\text{and}\quad 
 {\new \gamma < \po\mo - \varepsilon.} \label{crcgrowthinf}
\end{align}
In addition,  we choose~$n_0$ sufficiently large  such  that we have
for all $n \ge n_0$
\begin{align}
 \gamma^n > n, \quad 
 \gamma^n (\qc\mc-\varepsilon) > (n+1),\quad 
 \gamma^{\beta\log(\gamma^{n})} (\qc\mc-\varepsilon) \ge 1  
\label{crcpropgam}
\end{align}
for some $\beta>0$ which satisfies the assumptions of Lemma~\ref{crclem0.2}. In the following we again investigate the behaviour of the CBRW on the event $H_{n_0}$ on which the process does not show certain unlikely behaviour after time $n_0$ (cf.~\eqref{crc1a} and~\eqref{crc5a}). We show that already the offspring of one position with ``many" particles cause the leftmost cookie to move to the right with speed 1. For this, we introduce the random time
\[
\eta:= \inf \{n \geq n_0:\, \exists x \in \Z \text{ such that } |Z_n(x)| 
\ge \gamma^n\}.
\]
At time~$\eta$ we have sufficiently many particles at the random position $x_0:= \sup\{x\in \Z:\, Z_{\eta}(x)\ge \gamma^{\eta}\}$. Due to~\eqref{crcgrowthinf} we a.s.\ have $\eta<\infty$ since at time $n$ only $n+1$ positions can be occupied. 
Additionally, we introduce the random time
\[
 \sigma_0:= \inf\{n\ge\eta:\, l(n)=x_0 + n - \eta\}
\]
at which offspring of the particles belonging to $Z_{\eta}(x_0)$ can potentially reach the LP for the first time 
after time~$\eta$. Since $\pc\mc\le1$, the LP dies out infinitely often and therefore we a.s.\ have $\sigma_0<\infty$. Then, we define inductively the random times
\begin{align*}
&\tau_j:= \inf\{n \ge \sigma_j:\, |\lp{n}|=0\} - \sigma_j, \text{ for } j 
\geq 0,\\
&\sigma_j:= \inf\{n \ge \sigma_{j-1} + \tau_{j-1}:\, 
|\lp{n}| \neq 0\}, \text{ for } j \geq 1,
\end{align*}
denoting the time period of survival and the time of the restart of the LP after time $\sigma_0$. Due to~\eqref{crcpropgam} we have
\begin{equation} \label{crceq67}
|Z_{\eta}(x_0)| \ge \gamma^{\eta} \ge \eta
\end{equation}
which allows us to use the lower bound for $|Z^+_{\eta+1}(x_0+1)|$ on $H_{n_0}$. By using~\eqref{crcgrowthinf} and~\eqref{crceq67} we get on the event $H_{n_0} \cap \{l(\eta) > x_0\}$
\[
|Z_{\eta+1}(x_0+1)| \geq  |Z^+_{\eta+1}(x_0+1)| \geq  \gamma^{\eta} 
 (\po\mo-\varepsilon) \ge \gamma^{\eta+1}.
\]
Iterating the last step, we see that on the event
\[{\new H_{n_0}  
\cap \bigcap_{i=0}^{k-1}  \{l(\eta+k-1) > x_0 + k-1\} = H_{n_0}  
\cap  \{l(\eta+k-1) > x_0 + k-1\}}\] we have $|Z_{\eta+k}(x_0+k)| \geq \gamma^{\eta+k}$ and therefore we conclude that
\begin{equation*}
|\lp{\sigma_0}| = |Z_{\eta+\sigma_0 - \eta}(x_0+\sigma_0 - \eta)| 
\ge \gamma^{\eta+\sigma_0 - \eta} = \gamma^{\sigma_0}
\end{equation*}
holds on $H_{n_0}$. In the following we see that already the offspring particles of $\lp{\sigma_0}$ which move to the left at time $\sigma_0$ and afterwards move to the right in every step lead to a very large LP at the next restart at time $\sigma_1$. To see this, we first notice that~\eqref{crcpropgam} implies on the event $H_{n_0}$
\[
|Z_{\sigma_0+1}(l(\sigma_0)-1)| \ge |Z^-_{\sigma_0+1}(l(\sigma_0)-1)| 
\ge \gamma^{\sigma_0}  (\qc\mc-\varepsilon) \geq (\sigma_0+1)
\]
since we have $|Z_{\sigma_0}(l(\sigma_0))|\ge \gamma^{\sigma_0}>\sigma_{0}$.
 An iteration of this together with~\eqref{crcgrowthinf} and~\eqref{crcpropgam}  yield for $k \in \N$
\begin{align*}
  |Z_{\sigma_0+1+k}(l(\sigma_0)-1+k)|\ge\,& |Z^+_{\sigma_0+1+k}(l(\sigma_0)-1+k)|
 \\
 \ge\, &  \gamma^{\sigma_0}  (\qc\mc-\varepsilon) 
 (\po\mo-\varepsilon)^{k}\\
  \ge\, & {\new  \gamma^{\sigma_0+k}  }  (\qc\mc-\varepsilon) 
{\new \ge \sigma_0 + k+ 1 }
\end{align*}
on the event $H_{n_0} \cap \{\tau_0 \ge k-1\}$. 
In particular, this implies
\begin{align*}
|\lp{\sigma_0+\tau_0+2}|=\, &|Z_{\sigma_0+\tau_0+2}(l(\sigma_0)+\tau_0)| \\
\ge\, & \gamma^{\sigma_0+2(\tau_0+1)}  (\qc\mc-\varepsilon) \\
\ge \, & \gamma^{\sigma_0+\tau_0+2}  \gamma^{\beta 
\log(\gamma^{\sigma_0})}  (\qc\mc-\varepsilon) 
\ge \gamma^{\sigma_0+\tau_0+2} > 0
\end{align*}
on the event $H_{n_0}$. Here we used that, due to Lemma~\ref{crclem0.2}, we have $\tau_0 \ge \beta \log(\gamma^{\sigma_0})$ and recalled~\eqref{crcpropgam} for  the last inequality. Further, we conclude that we have $\sigma_1= \sigma_0 + \tau_0 + 2$ on $H_{n_0}$, which implies that the LP is restarted two time steps after it has died out at time $\sigma_0+\tau_0$. Iterating this argument finally implies 
\begin{align}
|\lp{\sigma_{j+1}}| \geq \gamma^{\sigma_{j+1}} &\quad\text{ and }\quad
\sigma_{j+1}= \sigma_j + \tau_j + 2  \label{crcgrolp}
\end{align}
for all $j \in \N_0$ on the event $H_{n_0}$. For $\beta^*:= \beta  \log(\gamma) > 0$ we further conclude from~\eqref{crcgrolp} and Lemma~\ref{crclem0.2} by induction that on $H_{n_0}$ we have for $j \in \N_0$
\begin{equation}\label{crcetst1}
 \tau_j \geq \beta  \sigma_j  \log(\gamma) \ge \beta^* 
 (1 + \beta^*)^{j}  \sigma_0 \end{equation}
and thus
\begin{equation}\label{crcetst2}
 \sigma_{j+1} = \sigma_j+\tau_j+2 \geq (1 + \beta^*)^j  
\sigma_0 +\beta^*  (1 + \beta^*)^{j}  \sigma_0  
= (1 + \beta^*)^{j+1}  \sigma_0. 
\end{equation}
Hence, on the event $H_{n_0}$ we have for $n \ge \sigma_0$
\begin{align*}
\frac{l(n)}{n} \ge\, & \frac{l(\sigma_0)+n-\sigma_0-2 |\{j\ge 0:\, 
\sigma_j + \tau_j \le n\}|}{n} 
\ge \frac{l(\sigma_0)+n-\sigma_0-2
 \frac{\log(n) - \log(\sigma_0)}{\log(1+\beta^*)}}{n} 
\xrightarrow[n \to \infty]{}  1.
\end{align*}
Here we use~\eqref{crcgrolp} in the first step and in 
the second step we use the fact that due to~\eqref{crcetst1} 
and~\eqref{crcetst2} we have $\sigma_j+\tau_j \ge (1+\beta^*)^{j+1} \sigma_0$
for $j \in \N_0$. This yields that on the event $H_{n_0}$ we have
$\lim_{n \to \infty} \frac{l(n)}{n} = 1$.
Since by~\eqref{crcprophn} we have $\lim_{n \to \infty} \p(H_n) = 1$,
we finally established~\eqref{crclem5.2}.
\renewcommand{\qedsymbol}{$\square$}
\end{proof}
With Proposition~\ref{crccor0} and Proposition~\ref{crcprop2}
 we are now prepared to prove Theorem~\ref{thm1}(b).
Similarly to the proof of Theorem~\ref{thm1}(a), 
we introduce the events
\[
A_n:=\{\exists \nu \succeq L_n:\, X_{\nu}=0,\, X_{\eta}<l(|\eta|)\, 
\forall\,  L_n\prec \eta \prec \nu\}
\]
with $L_n=\{\nu \in Z_{n+1}(l(n)-1):\, \nu \succ \lp{n}\}$ for $n \in \N$. 
On $A_n$, there exists a particle $\nu$ which returns to 
the origin after time $n$ and additionally the last ancestor of~$\nu$ 
which has been at a position containing a cookie was the ancestor at time~$n$. 
For $\lambda_0, \gamma >0$, which we will choose later (cf.~\eqref{crcgamma1}
 and~\eqref{crcgamma2}), we get the following estimate with 
$\mm=\max\{1,\po\mo\}$:
\begin{align*}
\p \Big(A_n \mid \{l(n)\ge n  \lambda_0\} \cap \{{\new |\lp{n}|} 
\le (\mm+\gamma)^n\}\Big) 
= \, & 1 - \p \Big(A_n^c \mid \{l(n)\ge n  \lambda_0\} 
\cap \{{\new |\lp{n}|} \le (\mm+\gamma)^n\}\Big) 
 \\
\le \, & 1 - \p \Big(\Lambda^-_{\lceil n \lambda_0-1\rceil} = 0 
\Big)^{M(\mm+\gamma)^n}.
\end{align*}
Here we use the fact that the number of offspring of every particle 
belonging to $L_n$ which return to the origin is bounded by the amount 
of offspring in $\Lambda^-_{l(n)-1}$. 
Additionally, we have $|L_n|\le M  |\lp{n}|$ 
due to assumption~\eqref{assm}. 
Since the GWP $\big(\Lambda^-_n\big)_{n \in \N_0}$ with mean 
$\fim$ is subcritical 
we can use Proposition \ref{propo1} to obtain for some constants $c, C>0$ and large $n$ that
\begin{align}
 \p \Big(A_n \cap \{l(n)\ge n  \lambda_0\} \cap \{{\new |\lp{n}|} 
\le (\mm+\gamma)^n\}\Big) 
\le \, & 1 - \big(1-c(\fim)^{\lceil n \lambda_0-1\rceil}
 \big)^{M (\mm+\gamma)^n} 
\nonumber \vphantom{\p \Big(\Lambda^-_{\lceil n \lambda_0-1\rceil} 
= 0 \Big)^{M (\mm+\gamma)^n} }\\
\le\, & 1 - \exp \left(-2  c(\fim)^{\lceil n \lambda_0-1\rceil}
     M (\mm+\gamma)^n \right) 
\nonumber \vphantom{\p \Big(\Lambda^-_{\lceil n \lambda_0-1\rceil} 
= 0 \Big)^{M (\mm+\gamma)^n} }\\
\le\, & 2  c(\fim)^{n \lambda_0-1}  
 M (\mm+\gamma)^n \nonumber \vphantom{\p 
\Big(\Lambda^-_{\lceil n \lambda_0-1\rceil} = 0
 \Big)^{M (\mm+\gamma)^n} } \\ 
=\, & C (\fim)^{n \lambda_0}  (\mm+\gamma)^n 
\label{crcaest1} \vphantom{\p \Big(\Lambda^-_{\lceil n \lambda_0-1\rceil}
 = 0 \Big)^{M (\mm+\gamma)^n} }.
\end{align}
In the above display we use the inequalities $1-x \ge \exp(-2x)$ for $x\in[0,\tfrac12]$ (note that we have $\fim<1$) and $1-\exp(-x) \le x$ 
for all $x \in \R$.

Let us first assume that we have $\mm = \max\{1,\po\mo\} = 1$. We choose $\lambda_0= \lambda/ 2$ for some $\lambda>0$ which satisfies the assumptions of Proposition~\ref{crcprop2}(a). 
We have $\fim \le 2\qo\mo < 1$ and therefore can choose $\gamma > 0$  such that
\begin{align}
(\fim)^{\lambda_0}  (\mm+\gamma)
 \le (2\qo\mo)^{\lambda_0}  (1+\gamma) \le (1-\gamma). 
\label{crcgamma1}
\end{align}
By applying \eqref{crcgamma1} to \eqref{crcaest1}, we get $ \p \Big(A_n \cap \{l(n)\ge n  \lambda_0\} \cap \{{\new |\lp{n}|} \le (\mm+\gamma)^n\}\Big) 
 \le  o(1)  (1-\gamma)^n$. Therefore, the Borel-Cantelli lemma implies
\begin{equation} \label{crcborc1}
\p \left(\limsup_{n \to \infty} 
\Big(A_n \cap \{l(n)\ge n  \lambda_0\} \cap \{{\new |\lp{n}|}
\le (\mm+\gamma)^n\}\Big)\right)=0.
\end{equation}
Moreover, Proposition~\ref{crccor0} and Proposition~\ref{crcprop2} together with the choices of~$\lambda_0$ and~$\gamma$ yield
\[
\p \left(\liminf_{n \to \infty} \Big(\{l(n)\ge n  \lambda_0\} 
\cap \{{\new |\lp{n}|} \le (\mm+\gamma)^n\}\Big)\right) = 1.
\]
Finally, we can conclude from~\eqref{crcborc1} that we have $\p \left(\limsup_{n \to \infty} A_n \right)=0$, which implies the transience of the CBRW in this case.

We now assume that we have $\mm = \po\mo > 1$. Due the assumption of the transience of the BRW without cookies, we have
\[
\fim \po\mo\le 2\qo\mo \cdot \po\mo\le\frac12.
\]
Therefore, we can choose $0<\gamma<1$ such that
\begin{equation} \label{crcgamma2}
(\fim)^{1-\gamma} (\po\mo+\gamma) \leq \frac34.
\end{equation} 
For $\lambda_0:=1-\gamma$, \eqref{crcaest1} and~\eqref{crcgamma2} imply $ \p \Big(A_n \cap \{l(n)\ge n  \lambda_0\} \cap \{{\new |\lp{n} |}\le (\mm+\gamma)^n\}\Big) \le C  \left(\frac34\right)^n$. Again by applying the Borel-Cantelli lemma, we get
\[
\p \left(\limsup_{n \to \infty} \big(A_n \cap \{l(n)\ge n 
 \lambda_0\} \cap \{{\new |\lp{n}|} \le (\mm+\gamma)^n\}\big)\right)=0.
\]
Additionally, Proposition~\ref{crccor0} and Proposition~\ref{crcprop2} together with the choices of~$\lambda_0$ and~$\gamma$ yield
\[
\p \left(\liminf_{n \to \infty} 
\big(\{l(n)\ge n  \lambda_0\} 
\cap \{{\new |\lp{n}|} \le (\mm+\gamma)^n\}\big)\right) = 1.
\]
Therefore, we conclude that we have $\p \left(\limsup_{n \to \infty} A_n \right)=0$, which implies the transience of the CBRW in the case $\po\mo>1$. {\new This finishes the proof of Theorem \ref{thm1}.}$\hfill\blacksquare$

\subsection*{Proof of Theorem~\ref{thm2}}
For this theorem we only have to make sure that the cookies cannot displace the cloud of particles too far to the right. It turns out that, somewhat similarly to the case of a cookie {\new (or excited)} random walk (cf.\ {\new Theorem 12 in \cite{Z05}}) one single cookie at every position $x \in \N_0$ is not enough for such a behaviour.

We divide the proof of the theorem into two cases. At first we consider the case $\mo=1$, i.e.\ particles can only branch at positions with a cookie, and in the second part we consider the case $\mo>1$.

Let us first assume $\mo=1$. In this case the BRW without cookies reduces to a nearest neighbour random walk on~$\Z$ and is therefore strongly recurrent iff $\po=1/2$ holds. {\new Further it is enough to only investigate the path of the first offspring particle in each step since already those particles will visit the origin infinitely often with probability 1. For $\pc \le 1/2$, the strong recurrence is obvious since we can bound the path of the considered particles from above by the path of a symmetric random walk on $\Z$ with the help of a coupling. For $1/2 < \pc < 1$, we can couple the random movement of the considered particles to a symmetric random walk and an excited random walk in the sense of \cite{BW} (with excitement $\varepsilon=2\pc - 1$) such that the position of the considered particles lies between the symmetric random walk (to the left) and the excited random walk (to the right). Since both random walks are recurrent (cf.\ Section 2 in \cite{BW} for the excited random walk), we can again conclude that the CBRW is strongly recurrent.}

 Now we suppose that we have $\mo > 1$. From Proposition~\ref{prop1.11} we know that we have $\log(\mo) > - \frac12 \log \left(4 \po \qo \right) = I(0)$ where $I(\cdot)$ denotes the rate function of the nearest neighbour random walk on~$\Z$ with transition probabilities~$\po$ and~$\qo$. Since the rate function is continuous on $(-1,1)$, there exist {\new $0<\varepsilon,\delta<1$} such that {\new $\log(\mo) > I(-\varepsilon) + \delta$}. Let $\big(S_n\big)_{n \in \N_0}$ denote such a nearest neighbour random walk started in~$0$ and with transition probabilities~$\po$ and~$\qo$. 
We have
\[
\lim_{n \to \infty} \frac1n\log\p\big(S_n \leq -n  \varepsilon \big) 
= \left. \begin{cases}-I(-\varepsilon) &\text{for }2\po-1>-\varepsilon 
\\ 0 &\text{for }2\po-1\le -\varepsilon 
\end{cases} \right\} \ge -I(-\varepsilon).
\]
{\new In particular,} there exists $k_0$ such that $\p \big(S_{k_0} \leq - k_0  \varepsilon \big) \geq \exp\big(-k_0(I(-\varepsilon)+\delta)\big)$. This yields for the BRW without cookies $\big(Y_{\nu}\big)_{\nu \in \mathcal{Y}}$ that
\begin{align}\label{thm2.2}
\E\Big[{\new \Big|\left\{\nu \in \mathcal{Y}:\ |\nu|=k_0,\ Y_{\nu} 
\leq -k_0 \varepsilon\right\} \Big|}\Big] \geq (\mo)^{k_0}
 \exp\big(-k_0 (I(-\varepsilon)+\delta)\big)>1.
\end{align}
for {\new our choice for $\delta, \varepsilon$}. Therefore, we can conclude that the embedded GWP of those particles which move at least $k_0 \varepsilon$ to the left between time $0$ and $k_0$, between $k_0$ and $2k_0$ and so on is supercritical and therefore survives with strictly positive probability $p_{\textnormal{sur}}$. Let us now turn back to the CBRW. For every existing particle {\new $\nu$} the probability
\[
\p \big(\exists \eta \in \mathcal{Z}:\ \eta \succeq \nu,\ 
|\eta|-|\nu|=k_0,\ X_{\eta}=X_{\nu}-k_0\mid  \nu \in \mathcal{Z} \big) 
\geq \min(\qc,\qo) \qo^{k_0}
\]
{\new  to have a descendant $k_0$ generations later which is located $k_0$ positions to the left of the position of $\nu$}\ is bounded away from 0. From this we conclude that the probability
\begin{align} \label{thm2.1}
\p \big(\exists \eta \in \mathcal{Z}:\ \eta \succeq \nu,\ X_{\tau} 
\leq l(|\nu|)\ \forall \nu \preceq\tau\preceq\eta,\ X_{\eta} 
\leq 0 \mid  \nu \in \mathcal{Z} \big)  
\geq\,&\qc \qo^{k_0} p_{\textnormal{sur}}=:c>0
\end{align}
for every existing particle in the CBRW {\new  to have a descendant on the negative semi-axis without any cookie contact in the ancestral line connecting $\nu$ and its descendant } is also bounded away from 0. Here the lower bound is a lower estimate for the probability for each existing particle {\new  $\nu$ in the CBRW to have a descendant $k_0$ generations later which is located $k_0$ positions to the left of the position of $\nu$ } and then starts a surviving embedded GWP which moves at least $k_0 \varepsilon$ to the left between time $0$ and $k_0$, between $k_0$ and $2k_0$ and so on. 
Since the particles we consider for this embedded GWP cannot hit the cookies in between, this GWP has the same probability for survival $p_{\textnormal{sur}}$ as in the case of the BRW without cookies (cf.\ \eqref{thm2.2}). Using~\eqref{thm2.1} we can conclude the strong recurrence of the CBRW since the particles on the negative semi-axis behave as the strongly recurrent BRW without cookies before they can reach a cookie again.$\hfill\blacksquare$

\subsection*{{\new Proof of} Theorem~\ref{thm3}}
\textbf{Proof of part (a).}
Here, we suppose that the LP is supercritical, i.e.\ $\pc\mc>1$. On the one hand the probability that all particles which are produced in the first step move to the left and their offspring then escape to $-\infty$ without returning to 0 is strictly positive since every offspring particle starts an independent BRW without cookies at position~$-1$ as long as the offspring does not return to the origin. We note that the probability for the BRW started at~$-1$ never to return to the origin is strictly positive since the BRW without cookies is transient to the left by assumption.

On the other hand the LP which is started at 0 is a supercritical GWP and therefore survives with positive probability. If it survives, a.s.\ infinitely many particles leave the LP (to the left) at time $n\ge1$. Afterwards each of those particles starts a BRW without cookies at position $n-1\geq 0$ since the offspring cannot reach a cookie again. Each of those BRWs without cookies will a.s.\ produce at least one offspring which visits the origin since the BRW without cookies is transient to the left by assumption. $\hfill\blacksquare$

\medskip
\noindent
\textbf{Proof of part (b).}
Here, we suppose that the LP is critical or subcritical,
 i.e.\ $\pc\mc\le1$.
In the following we want to consider the following three quantities.
 The first one is the {\new number} of particles in the LP. 
The second one is the number of particles which are descendants of the non-LP 
particles of generation $n$ {\new (i.e.\ $Z_n\setminus\lp{n}$)} and which are the first in their ancestral 
line to reach the position $l(n)$. 
{\new By definition,} these particles can potentially change the position {\new $l(n)$} of the leftmost cookie {\new in the future}.
 The third {\new quantity} is the number of {\new particles belonging to $Z_n\setminus\lp{n}$ whose descendants will not reach the position $l(n)$ in the future}. More precisely, for all $n \in \N_0$ we define
\begin{align*}
\zeta_1(n):= \ & \big|\lp{n}\big| ,\\
\zeta_2(n):= \ & {\new \big| \big\{ \nu\succeq Z_n\setminus\lp{n}:\ X_\nu
 = l(n),\, X_\eta <l(n)\, \forall\eta\prec\nu \big\} \big|},\\
\zeta_3(n):= \ & {\new \big| \big\{ \nu \in Z_n\setminus\lp{n}:\ X_\eta 
< l(n),\, \forall\eta\succeq\nu \big\} \big| }.
\end{align*}
Note that for the definition of $\zeta_2(n)$ we count the {\new number} of 
descendants of the non-LP particles at time~$n$ which will reach the 
position~$l(n)$ in the future. Thus the type-2 particles belong to a 
generation larger than~$n$.

In the following we want to allow arbitrary starting configurations from 
the set
\[
 \mathcal{S}:=\left\{(a,b)\in\N_0^\Z\times\N_0:\ \sum_{k\in\Z}a(k)<\infty,\,
\max\{k\in\Z:\,a(k)>0\}\leq b \right\}.
\]
Here $a$ contains the information about the number of particles at each position $k\in\Z$ and~$b$ is the position of the leftmost cookie. In particular, every configuration of the CBRW which can be reached within finite time is contained in the set~$\mathcal{S}$. For each $(a,b)\in\mathcal{S}$ we consider the probability measure~$\p_{(a,b)}$ under which the CBRW starts in the configuration $(a,b)$ and then evolves in the usual way.

The main idea of the proof is the following. We show that there is a critical level for the total amount of the type-1 and type-2 particles. Once this level is exceeded the total amount has the tendency to fall back below this level. There are two reasons which cause this behaviour. On one hand, the expected amount of type-2 particles which stay type-2 particles for another time step decreases every time the leftmost cookie is consumed by a type-1 particle. On the other hand, if there are many type\hbox{-}1 particles, the LP survives for a long time with high probability and meanwhile the remaining particles have time to escape to the left.

For the proof we have to analyse the relation between the type-1 and type-2 particles and to distinguish between two different situations. In the first one, there are type-1 particles at time~$n$ and therefore the leftmost cookie is consumed. In the second case there are no type-1 particles and therefore the position of the leftmost cookie does not change. Let us first assume that there are type-1 particles at time~$n$. Then, on the event $\{\zeta_1(n)\neq0\}$ we a.s.\ have
\begin{align}
& \E_{(a,b)} \big[ \zeta_1(n+1) \mid \zeta_1(n), \zeta_2(n) \big] 
 = \zeta_1(n)  \pc\mc, \nonumber \\
& \E_{(a,b)} \big[ \zeta_2(n+1) \mid \zeta_1(n), \zeta_2(n) \big]  
= \zeta_1(n)  \qc\mc (\fip)^2 + \zeta_2(n)  \fip.  
\label{calc2}
\end{align}
Here the last equality holds since each type-1 particle produces an expected amount of $\qc\mc$ particles which leave the LP to the left. To decide {\new how large their contribution to the type-2 particles counted at time $n+1$ is in expectation}, we have to count the number of their offspring which will reach position $l(n+1)=l(n)+1$ in the future. For each of these particles the distribution of this random number coincides with the distribution of $\Lambda^+_2$ whose expectation is given 
by $(\fip)^2$. 
Additionally, since one cookie is consumed the amount of type-2 particles, which are still type-2 particles at time $n+1$, decreases in expectation by $\fip$. Observe that due to the transience to the left of the BRW without cookies, the process $\big(\Lambda_n^+\big)_{n \in \N_0}$ is a GWP with mean $\fip < 1$ (cf. {\new \eqref{prop1.11a.r2} and \eqref{prop1.11a.r2.2}}).

Let us now assume that the LP is empty. Then, on $\{\zeta_1(n)=0\}$ we a.s.\ have
\begin{align}
 \E_{(a,b)} \big[ \zeta_1(n+1) + \zeta_2(n+1) \mid  \zeta_1(n), 
\zeta_2(n) \big] = \zeta_2(n),  \label{calc3}
\end{align}
since the position of the leftmost cookie does not change, i.e.\ $l(n+1)=l(n)$. Therefore, each type-2 particle of time~$n$ either still is a type-2 particle at time~$n+1$ or becomes a type-1 particle.

First, we deal with the subcritical case, i.e.\ $\pc\mc<1$. For fixed $h\in\N$ (which will be specified later, cf.~\eqref{c4}) we define the following random times
\begin{align*}
 \eta_{n+1}&:=\begin{cases}
(\eta_{n}+ h)\wedge \inf\{i>\eta_{n}:\zeta_1(i)=0\},&\text{if }
\zeta_1(\eta_{n})>0,\\
 (\eta_{n}+ h)\wedge \inf\{i>\eta_{n}:\zeta_1(i)>0\},&\text{if }
\zeta_1(\eta_{n})=0,
         \end{cases}
\end{align*}  
for $n\in\N_0$ and $\eta_0:=0$. We note that we have $\eta_{n+1}-\eta_n\le h$. 
For $n\in\N_0$ we define
\[
 \xi_1(n):=\zeta_1(\eta_n), \quad
\xi_2(n):=\zeta_2(\eta_n)
\]
as the amount of type-1 and type-2 particles along the sequence $(\eta_n)_{n\in\N_0}$ and the associated filtration $\mathcal{F}_n:=\sigma\big(\xi_1(i),\ \xi_2(i),\ \eta_i:\ i\le n\big)$. We want to adapt Theorem~2.2.1 of~\cite{Menshikov} 
and start with the following lemma:
\begin{lem}\label{c2}
For suitable (large) $h,u\in\N$ we have
\begin{equation}\label{c3}
\E_{(a,b)}[\xi_1(n+1) + \xi_2(n+1)\mid \mathcal{F}_n] \leq \xi_1(n) + \xi_2(n) 
\end{equation} 
a.s.\ on $\{\xi_1(n)+\xi_2(n) \geq u\}$ for all $(a,b)\in\mathcal{S}$.
\end{lem}
\begin{proof}[Proof of Lemma~\ref{c2}]
Let us fix $(a,b)\in\mathcal{S}$. We choose $h\in\N$ large enough such that
\begin{equation}\label{c4}
 \big(\pc\mc\big)^h+\qc\mc \sum_{i=0}^{h-1}\big(\pc\mc\big)^i
 (\fip)^{h-i+1}<\frac12
\end{equation}
and
\begin{equation}\label{c5}
 (\fip)^{h}<\frac12.
\end{equation}
Such a choice is possible since $\pc\mc<1$ and $\fip<1$. 
 Then, we fix $c=c(h)$ such that
\begin{equation}\label{c6}
0<c\le\frac1{M^h} \left(1-\fip\right)
\end{equation}
holds. Recall that the particles in the LP constitute a subcritical GWP. Let $\big(GW_n^{\textnormal{sub}}\big)_{n\in\N_0}$ denote such a GWP (with the same offspring distribution). Then, for every $\delta>0$ there is $u=u(\delta,h,c)\in\N$ such that
\begin{equation}\label{c7}
 \p_{\left\lfloor uc/(c+1) \right\rfloor} \left( GW_h^{\textnormal{sub}}
\ge 1 \right)\ge1-\delta
\end{equation}
since the probability for each existing particle to have at least one offspring which moves to the right is strictly positive.

We now verify \eqref{c3} separately on the following three events:
\begin{align*}
 A_1&:=\{\xi_1(n)+\xi_2(n)\ge u\}\cap\{ \xi_1(n)=0 \},\\
 A_2&:=\{\xi_1(n)+\xi_2(n)\ge u\}\cap\{ 0<\xi_1(n)\le c \xi_2(n) \},\\
 A_3&:=\{\xi_1(n)+\xi_2(n)\ge u\}\cap\{ \xi_1(n)> c \xi_2(n) \}.
\end{align*}
Note that $A_1\cup A_2\cup A_3=\{\xi_1(n)+\xi_2(n)\ge u\}$.

On $A_1$ there is no particle in the LP between time $\eta_n$ and time $\eta_{n+1}$ by definition. Thus, the position of the leftmost cookie does not change during this period. Hence we a.s.\ have 
\[
 \E_{(a,b)}[ \xi_1(n+1)+\xi_2(n+1)\,|\,\mathcal{F}_n] \mathds{1}_{A_1}
=\xi_2(n) \mathds{1}_{A_1}
\]
due to \eqref{calc3}.

On $A_2$ there is at least one particle in the LP and thus the leftmost cookie is consumed at time $\eta_n$. 
Using $\eta_{n+1}-\eta_n\leq h$ and the fact that the total number of offspring of each particle is bounded by~$M$, we a.s.\ obtain on~$A_2$
 \[
\E_{(a,b)}[ \xi_1(n+1)+\xi_2(n+1)\,|\,\mathcal{F}_n] 
\le \big( \xi_1(n) M^h+ \fip \xi_2(n)\big)
\le \xi_2(n) \big( c M^h +\fip \big)
\le \xi_2(n).
\] 
Here we use \eqref{c6} in the last step.

Next, recall that $L_n=\left\{\nu\in Z_{n+1}(l(n)-1): \nu\succ \lp{n}\right\}$ denotes the {\new set} of particles which leave the leading process to the left at time $n$. Using \eqref{calc2} we a.s.\ get on the event~$A_3$
\begin{align*}
\lefteqn{    \E_{(a,b)}[ \xi_1(n+1)+\xi_2(n+1)\,|\,\mathcal{F}_n] }
\\
&= \E_{(a,b)}\big[ \big(\xi_1(n+1)+\xi_2(n+1)\big)
\mathds{1}_{\{ \eta_{n+1}-\eta_n<h \}}\mid \mathcal{F}_n\big] 
 \\
 &\hspace{0.5cm}
+\E_{(a,b)}\big[ \big(\xi_1(n+1)+\xi_2(n+1)\big)
 \mathds{1}_{\{ \eta_{n+1}-\eta_n=h \}}\mid 
\mathcal{F}_n\big] \\
&\le \left(M^{h-1} \xi_1(n)+\fip \xi_2(n)\right)
 \E_{(a,b)}\left[\mathds{1}_{\{ \eta_{n+1}-\eta_n<h \}}\mid 
\mathcal{F}_n\right]
+(\fip)^h \xi_2(n) \E_{(a,b)}\left[
\mathds{1}_{\{ \eta_{n+1}-\eta_n=h \}}\mid \mathcal{F}_n\right] 
\vphantom{\sum_{i=0}^{h-1}}\\
      &\hspace{0.5cm}+ \E_{(a,b)}\left[ |\lp{\eta_n+h}| 
  \mathds{1}_{\{ \eta_{n+1}-\eta_n=h \}}\,\Big|\,\mathcal{F}_n\right]
 \vphantom{\sum_{i=0}^{h-1}}\\
&\left.\hspace{0.5cm}+\ \sum_{i=0}^{h-1} \E_{(a,b)}\hspace{-2pt} 
\left[\sum_{\nu\succeq L_{\eta_n+i}}\mathds{1}_{\{X_{\nu}=l(\eta_n)+h,\, 
X_\eta<l(\eta_n)+h\,\forall\eta\prec\nu\}} 
 \mathds{1}_{\{ \eta_{n+1}-\eta_n=h \}} \,\right|\mathcal{F}_n \right]
\hspace{-0.61pt} \vphantom{\sum_{i=0}^{h-1}}. 
\end{align*}
Here in the second step we use that on the event $\{\eta_{n+1}-\eta_n<h\}$ (in expectation) the proportion at most~$\fip$ of the type-2 particles does not escape to the left since at least one cookie is consumed. On the event $\{\eta_{n+1}-\eta_n=h\}$ we consider three summands. The first corresponds to the type-2 particles at time~$\eta_n$ that are still type-2 particles at time $\eta_{n+1}$. The second corresponds to the particles that are still in the LP at time $\eta_{n+1}$ and the third to the particles which have left the LP in the meantime. Using~\eqref{calc2} and the fact that we have at least $\lfloor u c/(c+1)  \rfloor$ type-1 particles on the event~$A_3$, we continue the calculation and obtain that on the event~$A_3$ we a.s.\ have
\begin{align*}
 \E_{(a,b)}[ \xi_1(n+1)+\xi_2(n+1)\,|\,\mathcal{F}_n] 
\le&\ \Bigg[ \left(M^{h-1} \xi_1(n)+\fip \xi_2(n)\right)
 \p_{\left\lfloor u c/(c+1) \right\rfloor} 
\left( GW_h^{\textnormal{sub}} = 0 \right) \\
      &\hspace{0.5cm}+(\fip)^h \xi_2(n) 
+\big(\pc\mc\big)^h \xi_1(n) 
 \\
       &\hspace{0.5cm}
+\sum_{i=0}^{h-1}\xi_1(n) (\pc\mc)^i
 (\qc\mc)
  (\fip)^{h-i+1}\Bigg] 
\\
\le&\   \left(M^{h-1} \delta + \frac12\right)
 \xi_1(n)+\left(\fip \delta+\frac12\right)
  \xi_2(n) \\
\le&\ \xi_1(n)+\xi_2(n)
\end{align*}
for $\delta=\delta(M,h,\fip)$ sufficiently small. Here we use~\eqref{c4}, \eqref{c5}, and~\eqref{c7} for the latter estimates.
\renewcommand{\qedsymbol}{$\square$}
\end{proof}

We now turn to the case when we have a critical leading process, i.e., $\pc\mc=1$. Again for some $c>0$, which we specify later (cf.\ \eqref{condition c}), we inductively define the following random times 
\begin{align*}
\eta_{n+1}:= \begin{cases} \eta_n + 1 ,& \text{if }\zeta_2(\eta_n) 
\geq c  \zeta_1(\eta_n), \\
\inf \{ n > \eta_n: \ \zeta_1(n)=0 \}, & \text{if }\zeta_2(\eta_n) < c 
 \zeta_1(\eta_n),
\end{cases}
\end{align*}
for $n \in \N_0$ and $\eta_0  := 0$. Similarly to above, we define 
for $n \in \N_0$
\[
\xi_1(n)  := \zeta_1(\eta_n), \quad
\xi_2(n)  := \zeta_2(\eta_n)
\]
and the associated filtration 
$\mathcal{F}_n:=\sigma\big(\xi_1(i),\ \xi_2(i), \eta_i:i\le n\big)$.
Analogously to Lemma~\ref{c2}, we continue with the following 
\begin{lem}\label{c8.1}
For suitable (large) $u \in \N$ we have 
\begin{equation}\label{c9}
\E_{(a,b)}[\xi_1(n+1) + \xi_2(n+1)|\mathcal{F}_n]  \leq \xi_1(n) + \xi_2(n)
\end{equation}
a.s.\ on $\{\xi_1(n) + \xi_2(n) \geq u\}$ for all $(a,b)\in\mathcal{S}$.
\end{lem}
\begin{proof}[Proof of Lemma \ref{c8.1}]
Let us fix $(a,b)\in\mathcal{S}$. Again for some $u=u(c) \in \N$, which we specify later (cf.~\eqref{condition u}), we introduce the following sets
\begin{align*}
A_1 & := \{ \xi_1(n) + \xi_2(n) \geq u \} \cap \{\xi_2(n)
\geq c  \xi_1(n)  \}, \\
A_2 & := \{ \xi_1(n) + \xi_2(n) \geq u \} \cap \{\xi_2(n) 
< c  \xi_1(n)  \}.
\end{align*}
and show \eqref{c9} on the sets $A_1$ and $A_2$ separately.

On $A_1$ we a.s.\ have
\begin{align*}
 \E_{(a,b)}[\xi_1(n+1) + \xi_2(n+1)|\mathcal{F}_n]
\leq \ & \mathds{1}_{\{ \xi_1(n)= 0 \}}   \xi_2(n) 
+ \mathds{1}_{\{ \xi_1(n) > 0 \}}  \left( \fip  \xi_2(n) 
+ M  \xi_1(n) \right) \\
\leq \ & \mathds{1}_{\{ \xi_1(n)= 0 \}}  \xi_2(n) 
+ \mathds{1}_{\{ \xi_1(n) > 0 \}}  \left( \fip
 \xi_2(n) + M  c^{-1}   \xi_2(n) \right)
\\
\leq \ & \left[ \mathds{1}_{\{ \xi_1(n)= 0 \}}  
+ \mathds{1}_{\{ \xi_1(n) > 0 \}}  \left( \fip  + M
           c^{-1}  \right)\right]  \xi_2(n)
\\
\leq \ & \xi_2(n)
\end{align*}
for any
\begin{equation}
\label{condition c}
0<c \leq M  \left(1 - \fip \right)^{-1}.
\end{equation}
Here we use that on the event $A_1$ we have $\eta_{n+1} = \eta_n + 1$. If $\xi_1(n)= 0$ holds, then no cookie is eaten at time $\eta_n$ and therefore we have $\xi_2(n+1) = \xi_2(n)$. If $\xi_1(n)> 0$ holds, the leftmost cookie is consumed and therefore in expectation the amount of the type-2 particles is reduced by the factor~$\fip$.

Next, to investigate the behaviour on the event~$A_2$, consider first the case $(\xi_1(n),\xi_2(n))=(v,0)$ with $v\in\N$. From this we can easily derive the general case later on since each time a cookie is consumed the number of type-2 particles is reduced by the factor $\fip<1$. Therefore, the type-2 particles do not essentially contribute to the growth of the process. We have:
\begin{align}\label{c10}
\lefteqn{ \E_{(a,b)}[\xi_1(n+1) + \xi_2(n+1)\mid \mathcal{F}_n] 
 \1{(\xi_1(n),\xi_2(n))=(v,0)} }\nonumber \\
&= \E_{(a,b)}[\xi_2(n+1)|\mathcal{F}_n]  \1{(\xi_1(n),\xi_2(n))=(v,0)} 
\nonumber \\
&= \Big(\E_{(a,b)}\left[\xi_2(n+1)
 \mathds{1}_{\left\{\eta_{n+1}-\eta_n\le  v^{1/3} \right\}}
\mid\mathcal{F}_n\right] 
\nonumber \\
&\qquad+ \sum_{j> v^{1/3}} \E_{(a,b)}\left[\xi_2(n+1)
 \mathds{1}_{\{\eta_{n+1}-\eta_n=j\}}\mid
\mathcal{F}_n\right]\Big) \1{(\xi_1(n),\xi_2(n))=(v,0)} 
\end{align}
We now consider the first summand in~\eqref{c10}. For this we define
\begin{align*}
 E^0&:=\left\{ \max_{\ell=1,\ldots,\lfloor v^{1/3}\rfloor}
\zeta_1(\eta_n+\ell)\le v^{2/3}\right\},\\
 E^k&:=\left\{ \max_{\ell=1,\ldots,\lfloor v^{1/3}\rfloor}
\zeta_1(\eta_n+\ell)\in \big(2^{k-1} v^{2/3},2^{k} v^{2/3}\big]\right\}
\quad\text{for }k\ge1,
\end{align*}
in order to control the maximum number of particles in the LP. Using these definitions, we write
\begin{align}\label{c11}
  \left.\E_{(a,b)}\left[\xi_2(n+1)
 \mathds{1}_{\left\{\eta_{n+1}-\eta_n\le v^{1/3} \right\}}
\;\right|\;\mathcal{F}_n\right]
  =&\ \sum_{k=0}^\infty \left.\E_{(a,b)}\left[\xi_2(n+1)
         \mathds{1}_{E^k\cap\left\{\eta_{n+1}-\eta_n\le v^{1/3} 
\right\}}\;\right|\;\mathcal{F}_n\right]\nonumber\\
\le&\ \left.v^{1/3} M v^{2/3} \p_{(a,b)}
\left(\eta_{n+1}-\eta_n\le v^{1/3}\;\right|\;\mathcal{F}_n\right) 
\nonumber\\
&\quad+\sum_{k=1}^\infty v^{1/3} M 2^k v^{2/3}
 \p_{(a,b)}\big(\mathcal{B}_k(n,v)\mid \mathcal{F}_n\big),
\end{align}
where we use the notation
\[
  \mathcal{B}_k(n,v):=\left\{\exists \ell\in\{\eta_n+1,\ldots,\eta_{n+1}\}:\, 
\zeta_1(\ell)>2^{k-1} v^{2/3},\eta_{n+1}-\eta_n\le v^{1/3}\right\}.
\]
Here we note that each particle that leaves the LP starts a new BRW without cookies (as long as the offspring particles do not reach a cookie again) which is transient to the left by assumption. Thus, for each of those particles the expected number of descendants which reach the position $l(\eta_{n+1})$ (and therefore are type-2 particles at time $\eta_{n+1}$) is less than one since they have to move at least two steps to the right. Now we observe that on the event $\{(\xi_1(n),\xi_2(n))=(v,0)\}$ we a.s.\ have
\begin{align}\label{c11a}
&\left.\p_{(a,b)}\left(\eta_{n+1}-\eta_n\le v^{1/3}\;\right|\;\mathcal{F}_n\right)
=\p_v\left(T^{\textnormal{cr}}\le v^{1/3}\right)
\intertext{and}\label{c11b}
 &\left.\p_{(a,b)}\big(\mathcal{B}_k(n,v)\;\right|\;\mathcal{F}_n\big)\le v^{1/3}
 \p_{\left\lceil2^{k-1}v^{2/3}\right\rceil}\left(T^{\textnormal{cr}}
\le v^{1/3}\right)
\end{align}
where $T^{\textnormal{cr}}$ denotes the extinction time of a critical GWP whose offspring distribution is given by the number of particles produced by a single particle in the LP which stay in the LP. (Note that this coincides with the number of type-1 offspring of a type-1 particle.) Now we apply~\eqref{c11a}, \eqref{c11b} and Proposition~\ref{propo3} to \eqref{c11} and a.s.\ obtain
\begin{align}\label{c11c}
 \lefteqn{ \left.\E_{(a,b)}\left[\xi_2(n+1) \mathds{1}_{
\left\{\eta_{n+1}-\eta_n\le v^{1/3} \right\}}\right|\mathcal{F}_n\right]
             \1{(\xi_1(n),\xi_2(n))=(v,0)} }
\nonumber\\
&\le \left[M v \exp\left(-C \frac{v}{v^{1/3}}\right)
\right.
\left.+\sum_{k=1}^\infty M 2^k v^{4/3}
             \exp\left(-C \frac{2^{k-1} v^{2/3}}{v^{1/3}}
\right)\right]
             \1{(\xi_1(n),\xi_2(n))=(v,0)}\nonumber\vphantom{\sum_{k=0}^\infty}\nonumber\\
&  = o(v) \1{(\xi_1(n),\xi_2(n))=(v,0)} \vphantom{\sum_{k=0}^\infty}
\end{align}
where $C>0$ is the constant of Proposition~\ref{propo3}.

Now we deal with the second summand in \eqref{c10}. For some $\delta\in(0,\tfrac13)$ and $j\in\N$ we introduce the events
\begin{align*}
 F_j^0&:=\left\{ \max_{\ell=1,\ldots,\lfloor j^\delta \rfloor}\zeta_1
\big(\eta_n+j-\lfloor j^\delta\rfloor+\ell\big)\le j^{2\delta}\right\},\\
 F_j^k&:=\left\{ \max_{\ell=1,\ldots,\lfloor j^\delta \rfloor}
           \zeta_1\big(\eta_n+j-\lfloor j^\delta\rfloor+\ell\big)
\in\big(2^{k-1} j^{2\delta},2^k
                j^{2\delta}\big]\right\}\quad\text{for }k\ge1,
\intertext{and}
 G_j^0&:=\left\{ \max_{\ell=1,\ldots,j}\zeta_1(\eta_n+\ell)\le j^{1+\delta}
\right\},\\
 G_j^k&:=\left\{ \max_{\ell=1,\ldots,j}
           \zeta_1(\eta_n+\ell)\in\big(2^{k-1} j^{1+\delta},2^k
 j^{1+\delta}\big]\right\}\quad\text{for }k\ge1.
\end{align*}
On the events $G_j^k$ we control the maximum number of particles in the LP up to time~$j$, whereas on~$F_j^k$ we control the maximum number during the $\lfloor j^\delta\rfloor$ time steps before~$j$. We observe that on the event $F_j^k\cap G_j^\ell$ not more than $M\cdot2^\ell j^{2+\delta}$ particles leave the LP up to time $\eta_n+j-\lfloor j^{\delta}\rfloor$ (because of~$G_j^\ell$). Each of those particles starts a BRW without cookies and in average it contributes not more than $(\fip)^{\lfloor j^\delta \rfloor +1} \le (\fip)^{j^\delta}$ to the number of type\hbox{-}2 particles at time $\eta_n+j$. Similarly, on $F_j^k\cap G_j^\ell$ not more than $M 2^k j^{3\delta}$ particles leave the LP from time $\eta_n+j-\lfloor j^\delta\rfloor+1$ 
to time $\eta_n+j$ (because of~$F_j^k$). Further, it holds that each particle that leaves the LP starts a new BRW without cookies and for each of those particles the expected number of descendants which reach the position $l(\eta_{n+1})$ is less than one since they have to move at least two steps to the right. Thus we have
\begin{align}\label{c12}
\lefteqn{ \E_{(a,b)}\left[ \xi_2(n+1) \mathds{1}_{F_j^k\cap G_j^\ell
\cap\{\eta_{n+1}-\eta_n=j\}} \mid  \mathcal{F}_n \right]}\nonumber\\
& \left(M 2^\ell j^{2+\delta} \left(\fip\right)^{j^\delta}
+M 2^k j^{3\delta}\right)
                 \left.\p_{(a,b)}\left(F_j^k\cap G_j^\ell
\cap\left\{\eta_{n+1}-\eta_n=j\right\}\right|\mathcal{F}_n\right).
\end{align}
Now suppose that $\ell\ge k$ and $(k,\ell)\neq(0,0)$. 
Then due to Proposition~\ref{propo3} we have
\begin{align}\label{c13}
  \lefteqn{\p_{(a,b)}\left(F_j^k\cap G_j^\ell\cap\left\{\eta_{n+1}-\eta_n=j
\right\}\left|\;\mathcal{F}_n\vphantom{F_j^k}\right)\right.}
\nonumber \vphantom{\frac{2^{\ell-1} j^{1+\delta}}{j}}\\
 &\le \p_{(a,b)}\left(\exists i\in\{1,\ldots,j\}:\,
            \zeta_1(\eta_n+i)>2^{\ell-1} j^{1+\delta},\, 
\zeta_1(\eta_n+j)=0\, \left|\, \mathcal{F}_n\vphantom{2^{\ell-1} 
j^{1+\delta}}\right)\right.\nonumber \vphantom{\frac{2^{\ell-1} 
j^{1+\delta}}{j}}\\
 &\le j \p_{\left\lceil2^{\ell-1} j^{1+\delta} 
\right\rceil}\left(T^{\textnormal{cr}}\le j\right)
\nonumber \vphantom{\frac{2^{\ell-1} j^{1+\delta}}{j}}\\
 &\le j \exp\left(-\tfrac12C 2^{(\ell+k)/2}  j^\delta\right). 
\vphantom{\frac{2^{\ell-1} j^{1+\delta}}{j}}
\end{align}
If otherwise $k\ge\ell$ and $(k,\ell)\neq(0,0)$, then again due to Proposition~\ref{propo3} we have
\begin{align}\label{c14}
\lefteqn{\p_{(a,b)}\left(F_j^k\cap G_j^\ell\cap\left\{\eta_{n+1}-\eta_n=j\right\}
\left|\mathcal{F}_n\vphantom{F_j^k}\right)
        \right.}\nonumber \vphantom{\frac{2^{\ell-1} j^{1+\delta}}{j}}\\
&\le \p_{(a,b)}\left(\exists i \hspace{-1pt}\in\hspace{-1pt}\{j-\lfloor 
j^\delta\rfloor+1,\ldots,j\}:\,
            \zeta_1(\eta_n+i)>2^{k-1} j^{2\delta},\, \zeta_1(\eta_n+j)=0\,
 \left|\, \mathcal{F}_n\vphantom{2^{k-1} j^{2\delta}}\right)
        \right.\nonumber\vphantom{\frac{2^{\ell-1} j^{1+\delta}}{j}}
\displaybreak[0]\\
 &\le j \p_{\left\lceil2^{k-1} j^{2\delta} \right\rceil}
\left(T^{\textnormal{cr}}\le j^\delta\right)
            \nonumber \vphantom{\frac{2^{\ell-1} j^{1+\delta}}{j}}
\displaybreak[0]\\
 &\le j \exp\left(-\tfrac12C 2^{(\ell+k)/2}  j^\delta\right). 
\vphantom{\frac{2^{\ell-1} j^{1+\delta}}{j}}
\end{align}
With the help of~\eqref{c13} and~\eqref{c14} together with~\eqref{c12} we a.s.\ obtain
\begin{align}\label{c14a}
 \lefteqn{ \E_{(a,b)}\big[\xi_2(n+1) \mathds{1}_{\left\{\eta_{n+1}-\eta_n=j
\right\}}\mid \mathcal{F}_n\big]}\vphantom{\sum_{i=1}^\infty}\nonumber\\
&= \sum_{k,\ell=0}^\infty\left.\E_{(a,b)}\left[\xi_2(n+1)
 \mathds{1}_{F_j^k\cap G_j^\ell\cap\left\{\eta_{n+1}-\eta_n
=j\right\}}\right|\mathcal{F}_n\right]\vphantom{\sum_{i=1}^\infty}
\nonumber \displaybreak[0]\\
&\le  \left(M j^{2+\delta} \left(\fip\right)^{j^\delta}
+M j^{3\delta}\right) \p_{(a,b)}\left( \eta_{n+1}-\eta_n=j\, |\, 
\mathcal{F}_n\right)\vphantom{\sum_{(k,l)\neq(0,0)}}
\vphantom{\sum_{i=1}^\infty}\nonumber\\
&\quad+\sum_{(k,\ell)\neq(0,0)}\!
\left(M2^\ell j^{2+\delta}\! \!\left(\fip\right)^{j^\delta}\!
+M 2^k j^{3\delta}\right)
       j \exp\left(-\tfrac12C 2^{(\ell+k)/2} 
 j^\delta\right)\vphantom{\sum_{i=1}^\infty}\nonumber \displaybreak[0]\\
&\le {\new C_2 j^{3\delta} \p_{(a,b)}\left( \eta_{n+1}-\eta_n=j\,
 |\, \mathcal{F}_n\right)\vphantom{\sum_{i=1}^\infty}+ C_2 j^{1+3\delta} \sum_{i=1}^\infty (i+1) 2^i
       \exp\left(-\tfrac12C 2^{i/2} j^\delta\right)
\vphantom{\sum_{i=1}^\infty}}
\end{align}
{\new for some constant $C_2 > 0$ which does not depend on $j$.} By Proposition~\ref{propo4}, on~$\{(\xi_1(n),\xi_2(n))=(v,0)\}$ we a.s.\ have $\p_{(a,b)}\left( \eta_{n+1}-\eta_n= j\, |\, \mathcal{F}_n\right)\le C_3 \frac{v}{j^2}$ {\new (for some constant $C_3$)}, and therefore \eqref{c14a} yields
\begin{align}\label{c15}
 \lefteqn{ \E_{(a,b)}\left[\xi_2(n+1) \mathds{1}_{\left\{\eta_{n+1}
-\eta_n=j\right\}}\mid \mathcal{F}_n\right] \1{(\xi_1(n),\xi_2(n))=(v,0)}}
\nonumber\\
&\le \Bigg[ C_4 j^{3\delta-2} v \hspace{-1pt}+C_2 j^{1+3\delta}
 \exp\left(-\tfrac14C 2^{\frac12}  j^\delta\right)  
\sum_{i=1}^\infty(i+1)2^i 
 \exp\left(-\tfrac14C 2^{i/2} 1\right)\Bigg]
\vphantom{\Bigg[\sum_{i=1}^\infty\Bigg]}  \1{(\xi_1(n),\xi_2(n))=(v,0)}
\vphantom{\Bigg[\sum_{i=1}^\infty\Bigg]}\nonumber\\
&{\new \le \ } C_5 j^{3\delta-2} v \1{(\xi_1(n),\xi_2(n))=(v,0)} 
\vphantom{\sum_{m=1}^\infty}\vphantom{\Bigg[\sum_{i=1}^\infty\Bigg]}
\end{align}
{\new for suitable constants $C_4, C_5 > 0$}. Using the estimates \eqref{c11c} and \eqref{c15} for the two summands in~\eqref{c10}, we conclude
\begin{align*}
  \E_{(a,b)}[\xi_1(n+1) + \xi_2(n+1)|\mathcal{F}_n] 
 \1{(\xi_1(n),\xi_2(n))=(v,0)} 
 \le&\ \Big[o(v) +v \sum_{j>v^{1/3}} {\new C_5 j^{3\delta-2}}\Big]
 \1{(\xi_1(n),\xi_2(n))=(v,0)} \\
 =&\ v o(v) \1{(\xi_1(n),\xi_2(n))=(v,0)} ,
\end{align*}
and therefore there exists $v_0\in\N$ such that
\[
 \E_{(a,b)}[\xi_1(n+1) + \xi_2(n+1)\mid\mathcal{F}_n] 
 \1{(\xi_1(n),\xi_2(n))=(v,0)}\le v \1{(\xi_1(n),\xi_2(n))=(v,0)}
\]
for $v\ge v_0$.

For the general case, in which we can also have type-2 particles at time $\eta_n$, we notice that for
\begin{equation}\label{condition u}
u\ge(1+c) v_0
\end{equation}
we have
\[
 \E_{(a,b)}[\xi_1(n+1) + \xi_2(n+1)\mid\mathcal{F}_n]
  \mathds{1}_{A_2}\le\big[\xi_1(n) + \xi_2(n)\big] \mathds{1}_{A_2} 
\]
since on $A_2$ the type-2 particles which exist at time $\eta_n$ evolve independently of the LP until time $\eta_{n+1}$. {\new This finishes the proof of Lemma~\ref{c8.1}.}
\renewcommand{\qedsymbol}{$\square$}
\end{proof}
Now we fix $u\in\N$ such that Lemma~\ref{c2} and Lemma~\ref{c8.1} hold. Further, we define
\[
 \tau:=\inf\{n\in\N_0:\, \xi_1(n)+\xi_2(n)\le u\}.
\]
Due to Lemma~\ref{c2} and, respectively, Lemma~\ref{c8.1}, we see that in the subcritical (i.e.\ $\pc\mc<1$) as well as in the critical (i.e.\ $\pc\mc=1$) case $\big(\xi_1(n\wedge\tau) + \xi_2(n\wedge\tau)\big)_{n\in\N_0}$ is a non-negative supermartingale with respect to $(\mathcal{F}_n)_{n\in\N_0}$ and $\p_{(a,b)}$ for arbitrary $(a,b)\in\mathcal{S}$. Thus it converges $\p_{(a,b)}$-a.s.\ to a finite random variable $\mathcal{X}(a,b)$. Since we have $\xi_1(n\wedge\tau) + \xi_2(n\wedge\tau)
\in\N_0$ for all $n\in\N_0$ and since the probability for this process to eventually stay at a constant level $v>u$ for all times is equal to 0, we conclude that $\mathcal{X}(a,b)\le u$ holds $\p_{(a,b)}$-a.s. Therefore, for all $(a,b)\in\mathcal{S}$ we have
$\p_{(a,b)}\left(\exists n\in\N_0:\, \xi_1(n)+\xi_2(n)\le u\right)=1$, and hence
\begin{align} \label{c16}
 &\p_{(a,b)}\left(\exists n\in\N_0:\, \zeta_1(n)+\zeta_2(n)\le u\right)=1.
\end{align}
We now introduce the following random times
\begin{align*}
\sigma_i&:=  \inf \{n >\tau_{i}: \ l(n) = l(\tau_i) + 2 \}, &\text{for } 
i\geq 0, \\
\tau_i&:= \inf \{ n \geq \sigma_{i-1}: \ \zeta_1(n) + \zeta_2(n) \leq u \} ,
&\text{for } i\geq 1 ,
\end{align*}
with $\tau_0:=0$. Here~$\sigma_i$ denotes the first time at which two more cookies have been eaten since~$\tau_i$. Moreover, we observe that $\big(Y(n) \big)_{n \in \N_0} := \big(\big(Z_n(x)\big)_{x \in \Z}, l(n) \big)_{n \in \N_0}$
is a Markov chain with values in $\mathcal{S}$, which can only reach finitely (thus countably) many states within finite time. Therefore,~\eqref{c16} yields for $i \in \N_0$
\begin{equation} \label{c17}
\p_{(e_0,0)} \big(\tau_{i+1}< \infty \mid  \sigma_i < \infty \big)= 1
\end{equation}
where $(e_0,0)$ denotes the usual starting configuration with one particle and the leftmost cookie at position~$0$. Finally, we have
\begin{align} 
\label{c18}
 \p_{(e_0,0)} \big(\sigma_i= \infty  \mid   \tau_i < \infty  \big) 
\ge  \left(\qc \p(\Lambda^+_1=0)\right)^{M u}=:\gamma\in(0,1).
\end{align}
This inequality holds since at the first time after $\tau_i$, at which any particle reaches the leftmost cookie again, there are not more than $u$ type-1 particles. Each of those type-1 particles cannot produce more than~$M$ particles in the next step. 
Afterwards, the probability for any direct offspring of the type-1 particles to move to the left and then produce offspring 
which escape to $-\infty$ is given by $\qc \p(\Lambda^+_1=0)$. All the remaining type\hbox{-}2 particles escape to the left 
with probability $\p(\Lambda^+_1=0)$ since one more cookie has been eaten. In this case, only one more cookie is consumed after the random time $\tau_i$ implying $\sigma_i = \infty$.

Using~\eqref{c17} and~\eqref{c18} we can conclude
\begin{align*}
\p_{(e_0,0)} \big(\sigma_i < \infty \ 
\forall\ i \in \N \big) &\leq \p_{(e_0,0)} \big(\sigma_k < \infty \big) \\  
&=   \p_{(e_0,0)} \big(\sigma_0 < \infty \big)  \prod_{i=1}^{k} 
\p_{(e_0,0)} \big(\sigma_i < \infty  \mid   \tau_i < \infty\big) 
 \p_{(e_0,0)}\big(\tau_i < \infty \mid  \sigma_{i-1} < \infty\big) \\ 
&\leq  \big( 1 - \gamma \big)^k \xrightarrow[k \to \infty]{} 0.
\end{align*}
In particular this implies that a.s.\ only finitely many cookies are consumed and this yields that the CBRW is transient.$\hfill\blacksquare$

\section{Final remarks} 
\label{section6}
At the end, let us consider a CBRW with one cookie at every position $x\in\Z$, i.e.~$c_0(x):=1$ for all $x\in\Z$. In this case the leftmost cookie on the positive semi-axis
\begin{align*}
l(n)= & \min\{x\geq 0:\, c_n(x)=1\}
\intertext{and the rightmost cookie on the negative semi-axis}
r(n):= & \max\{x\leq 0:\, c_n(x)=1\}
\end{align*}
are of interest. With these two definitions we can introduce the right LP
$\lpr{n}:=Z_n(l(n))$, and the left LP
$\lpl{n}:=Z_n(r(n))$.
Using Theorems~\ref{thm1}, \ref{thm2}, and~\ref{thm3} and the symmetry of the CBRW with regard to the origin, one can derive the following results:
\begin{thm} 
\label{thm6.1}
Suppose that the BRW without cookies is transient to the right.
\begin{enumerate}
\item
If the right LP is supercritical, i.e.\ $\pc\mc > 1$ holds, then
\begin{enumerate}
\item
the CBRW is strongly recurrent iff $\pc\mc  \fim \ge 1$,
\item
the CBRW is weakly recurrent iff 
$\pc\mc  \fim < 1$ and $\qc\mc > 1$,
\item
the CBRW is transient {\new to the right} iff
$\pc\mc  \fim < 1$ and $\qc\mc \le 1$.
\end{enumerate}
\item
If the right LP is subcritical or critical, i.e.\ $\pc\mc \leq  1$ 
holds, then 
\begin{enumerate}
\item
the CBRW is weakly recurrent iff the left LP is supercritical, i.e.
$\qc\mc > 1$,
\item
the CBRW is transient {\new to the right} iff the left LP is subcritical or critical, i.e.
$\qc\mc \le 1$.
\end{enumerate}
\end{enumerate}
\end{thm} 

\begin{thm}
Suppose that the BRW without cookies is strongly recurrent. Then the CBRW is strongly recurrent, no matter which kinds of right and left LP we have.
\end{thm} 
\begin{thm}
Suppose that the BRW without cookies is transient to the left. Due to the symmetry of the process we get the same result as in 
Theorem~\ref{thm6.1} if we just replace right LP by left LP, $\pc$ by $\qc$, $\fim$ by $\fip$ and {\new ``to the right'' by ``to the left''}.
\end{thm}
{\new \section*{Acknowledgement:} 
We would like to thank an anonymous referee for carefully reading the first version of this paper and thoughtful and constructive comments.}

\end{document}